\documentclass[12pt]{amsart}
\usepackage{txfonts}      
\usepackage{amssymb}
\usepackage{eucal}
\usepackage{amsmath}
\usepackage{amscd}
\usepackage{xcolor}
\usepackage{multicol}
\usepackage[all]{xy}           
\usepackage{graphicx}
\usepackage{color}
\usepackage{colordvi}
\usepackage{xspace}
\usepackage{tikz}
\usepackage{makecell}
\usepackage{appendix}
\usepackage{amsthm}
\usepackage{graphicx}
\usepackage{tabularx}
\usepackage{float}

\usepackage{ifpdf}
\ifpdf
\usepackage[colorlinks,final,backref=page,hyperindex]{hyperref}
\else
\usepackage[colorlinks,final,backref=page,hyperindex,hypertex]{hyperref}
\fi

\usepackage[active]{srcltx} 

\begin{document}


\newtheorem{thm}{Theorem}[section]
\newtheorem{lem}[thm]{Lemma}
\newtheorem{cor}[thm]{Corollary}
\newtheorem{pro}[thm]{Proposition}
\theoremstyle{definition}
\newtheorem{defi}[thm]{Definition}
\newtheorem{ex}[thm]{Example}
\newtheorem{rmk}[thm]{Remark}
\newtheorem{pdef}[thm]{Proposition-Definition}
\newtheorem{condition}[thm]{Condition}

\renewcommand{\labelenumi}{{\rm(\alph{enumi})}}
\renewcommand{\theenumi}{\alph{enumi}}

\newcommand {\emptycomment}[1]{} 

\newcommand{\nc}{\newcommand}
\newcommand{\delete}[1]{}

\nc{\todo}[1]{\tred{To do:} #1}

\nc{\tred}[1]{\textcolor{red}{#1}}
\nc{\tblue}[1]{\textcolor{blue}{#1}}
\nc{\tgreen}[1]{\textcolor{green}{#1}}
\nc{\tpurple}[1]{\textcolor{purple}{#1}}
\nc{\tgray}[1]{\textcolor{gray}{#1}}
\nc{\torg}[1]{\textcolor{orange}{#1}}
\nc{\tmag}[1]{\textcolor{magenta}}
\nc{\btred}[1]{\textcolor{red}{\bf #1}}
\nc{\btblue}[1]{\textcolor{blue}{\bf #1}}
\nc{\btgreen}[1]{\textcolor{green}{\bf #1}}
\nc{\btpurple}[1]{\textcolor{purple}{\bf #1}}

    \nc{\mlabel}[1]{\label{#1}}  
    \nc{\mcite}[1]{\cite{#1}}  
    \nc{\mref}[1]{\ref{#1}}  
    \nc{\meqref}[1]{\eqref{#1}}  
    \nc{\mbibitem}[1]{\bibitem{#1}} 

\delete{
    \nc{\mcite}[1]{\cite{#1}{\small{\tt{{\ }(#1)}}}}  
    \nc{\mref}[1]{\ref{#1}{\small{\tred{\tt{{\ }(#1)}}}}}  
    \nc{\meqref}[1]{\eqref{#1}{{\tt{{\ }(#1)}}}}  
    \nc{\mbibitem}[1]{\bibitem[\bf #1]{#1}} 
}


\nc{\cm}[1]{\textcolor{red}{Chengming:#1}}
\nc{\yy}[1]{\textcolor{blue}{Yanyong: #1}}
\nc{\li}[1]{\textcolor{purple}{#1}}
\nc{\lir}[1]{\textcolor{purple}{Li:#1}}


\nc{\tforall}{\ \ \text{for all }}
\nc{\hatot}{\,\widehat{\otimes} \,}
\nc{\complete}{completed\xspace}
\nc{\wdhat}[1]{\widehat{#1}}

\nc{\ts}{\mathfrak{p}}
\nc{\mts}{c_{(i)}\ot d_{(j)}}

\nc{\NA}{{\bf NA}}
\nc{\LA}{{\bf Lie}}
\nc{\CLA}{{\bf CLA}}

\nc{\cybe}{CYBE\xspace}
\nc{\nybe}{NYBE\xspace}
\nc{\ccybe}{CCYBE\xspace}

\nc{\ndend}{pre-Novikov\xspace}
\nc{\calb}{\mathcal{B}}
\nc{\rk}{\mathrm{r}}
\newcommand{\g}{\mathfrak g}
\newcommand{\h}{\mathfrak h}
\newcommand{\pf}{\noindent{$Proof$.}\ }
\newcommand{\frkg}{\mathfrak g}
\newcommand{\frkh}{\mathfrak h}
\newcommand{\Id}{\rm{Id}}
\newcommand{\gl}{\mathfrak {gl}}
\newcommand{\ad}{\mathrm{ad}}
\newcommand{\add}{\frka\frkd}
\newcommand{\frka}{\mathfrak a}
\newcommand{\frkb}{\mathfrak b}
\newcommand{\frkc}{\mathfrak c}
\newcommand{\frkd}{\mathfrak d}
\newcommand {\comment}[1]{{\marginpar{*}\scriptsize\textbf{Comments:} #1}}

\nc{\vspa}{\vspace{-.1cm}}
\nc{\vspb}{\vspace{-.2cm}}
\nc{\vspc}{\vspace{-.3cm}}
\nc{\vspd}{\vspace{-.4cm}}
\nc{\vspe}{\vspace{-.5cm}}


\nc{\disp}[1]{\displaystyle{#1}}
\nc{\bin}[2]{ (_{\stackrel{\scs{#1}}{\scs{#2}}})}  
\nc{\binc}[2]{ \left (\!\! \begin{array}{c} \scs{#1}\\
    \scs{#2} \end{array}\!\! \right )}  
\nc{\bincc}[2]{  \left ( {\scs{#1} \atop
    \vspace{-.5cm}\scs{#2}} \right )}  
\nc{\ot}{\otimes}
\nc{\sot}{{\scriptstyle{\ot}}}
\nc{\otm}{\overline{\ot}}
\nc{\ola}[1]{\stackrel{#1}{\la}}

\nc{\scs}[1]{\scriptstyle{#1}} \nc{\mrm}[1]{{\rm #1}}

\nc{\dirlim}{\displaystyle{\lim_{\longrightarrow}}\,}
\nc{\invlim}{\displaystyle{\lim_{\longleftarrow}}\,}

\nc{\bfk}{{\bf k}} \nc{\bfone}{{\bf 1}}
\nc{\rpr}{\circ}
\nc{\dpr}{{\tiny\diamond}}
\nc{\rprpm}{{\rpr}}

\nc{\mmbox}[1]{\mbox{\ #1\ }} \nc{\ann}{\mrm{ann}}
\nc{\Aut}{\mrm{Aut}} \nc{\can}{\mrm{can}}
\nc{\twoalg}{{two-sided algebra}\xspace}
\nc{\colim}{\mrm{colim}}
\nc{\Cont}{\mrm{Cont}} \nc{\rchar}{\mrm{char}}
\nc{\cok}{\mrm{coker}} \nc{\dtf}{{R-{\rm tf}}} \nc{\dtor}{{R-{\rm
tor}}}
\renewcommand{\det}{\mrm{det}}
\nc{\depth}{{\mrm d}}
\nc{\End}{\mrm{End}} \nc{\Ext}{\mrm{Ext}}
\nc{\Fil}{\mrm{Fil}} \nc{\Frob}{\mrm{Frob}} \nc{\Gal}{\mrm{Gal}}
\nc{\GL}{\mrm{GL}} \nc{\Hom}{\mrm{Hom}} \nc{\hsr}{\mrm{H}}
\nc{\hpol}{\mrm{HP}}  \nc{\id}{\mrm{id}} \nc{\im}{\mrm{im}}

\nc{\incl}{\mrm{incl}} \nc{\length}{\mrm{length}}
\nc{\LR}{\mrm{LR}} \nc{\mchar}{\rm char} \nc{\NC}{\mrm{NC}}
\nc{\mpart}{\mrm{part}} \nc{\pl}{\mrm{PL}}
\nc{\ql}{{\QQ_\ell}} \nc{\qp}{{\QQ_p}}
\nc{\rank}{\mrm{rank}} \nc{\rba}{\rm{RBA }} \nc{\rbas}{\rm{RBAs }}
\nc{\rbpl}{\mrm{RBPL}}
\nc{\rbw}{\rm{RBW }} \nc{\rbws}{\rm{RBWs }} \nc{\rcot}{\mrm{cot}}
\nc{\rest}{\rm{controlled}\xspace}
\nc{\rdef}{\mrm{def}} \nc{\rdiv}{{\rm div}} \nc{\rtf}{{\rm tf}}
\nc{\rtor}{{\rm tor}} \nc{\res}{\mrm{res}} \nc{\SL}{\mrm{SL}}
\nc{\Spec}{\mrm{Spec}} \nc{\tor}{\mrm{tor}} \nc{\Tr}{\mrm{Tr}}
\nc{\mtr}{\mrm{sk}}

\nc{\ab}{\mathbf{Ab}} \nc{\Alg}{\mathbf{Alg}}

\nc{\BA}{{\mathbb A}} \nc{\CC}{{\mathbb C}} \nc{\DD}{{\mathbb D}}
\nc{\EE}{{\mathbb E}} \nc{\FF}{{\mathbb F}} \nc{\GG}{{\mathbb G}}
\nc{\HH}{{\mathbb H}} \nc{\LL}{{\mathbb L}} \nc{\NN}{{\mathbb N}}
\nc{\QQ}{{\mathbb Q}} \nc{\RR}{{\mathbb R}} \nc{\BS}{{\mathbb{S}}} \nc{\TT}{{\mathbb T}}
\nc{\VV}{{\mathbb V}} \nc{\ZZ}{{\mathbb Z}}


\nc{\calao}{{\mathcal A}} \nc{\cala}{{\mathcal A}}
\nc{\calc}{{\mathcal C}} \nc{\cald}{{\mathcal D}}
\nc{\cale}{{\mathcal E}} \nc{\calf}{{\mathcal F}}
\nc{\calfr}{{{\mathcal F}^{\,r}}} \nc{\calfo}{{\mathcal F}^0}
\nc{\calfro}{{\mathcal F}^{\,r,0}} \nc{\oF}{\overline{F}}
\nc{\calg}{{\mathcal G}} \nc{\calh}{{\mathcal H}}
\nc{\cali}{{\mathcal I}} \nc{\calj}{{\mathcal J}}
\nc{\call}{{\mathcal L}} \nc{\calm}{{\mathcal M}}
\nc{\caln}{{\mathcal N}} \nc{\calo}{{\mathcal O}}
\nc{\calp}{{\mathcal P}} \nc{\calq}{{\mathcal Q}} \nc{\calr}{{\mathcal R}}
\nc{\calt}{{\mathcal T}} \nc{\caltr}{{\mathcal T}^{\,r}}
\nc{\calu}{{\mathcal U}} \nc{\calv}{{\mathcal V}}
\nc{\calw}{{\mathcal W}} \nc{\calx}{{\mathcal X}}
\nc{\CA}{\mathcal{A}}

\nc{\fraka}{{\mathfrak a}} \nc{\frakB}{{\mathfrak B}}
\nc{\frakb}{{\mathfrak b}} \nc{\frakd}{{\mathfrak d}}
\nc{\oD}{\overline{D}}
\nc{\frakF}{{\mathfrak F}} \nc{\frakg}{{\mathfrak g}}
\nc{\frakm}{{\mathfrak m}} \nc{\frakM}{{\mathfrak M}}
\nc{\frakMo}{{\mathfrak M}^0} \nc{\frakp}{{\mathfrak p}}
\nc{\frakS}{{\mathfrak S}} \nc{\frakSo}{{\mathfrak S}^0}
\nc{\fraks}{{\mathfrak s}} \nc{\os}{\overline{\fraks}}
\nc{\frakT}{{\mathfrak T}}
\nc{\oT}{\overline{T}}
\nc{\frakX}{{\mathfrak X}} \nc{\frakXo}{{\mathfrak X}^0}
\nc{\frakx}{{\mathbf x}}
\nc{\frakTx}{\frakT}      
\nc{\frakTa}{\frakT^a}        
\nc{\frakTxo}{\frakTx^0}   
\nc{\caltao}{\calt^{a,0}}   
\nc{\ox}{\overline{\frakx}} \nc{\fraky}{{\mathfrak y}}
\nc{\frakz}{{\mathfrak z}} \nc{\oX}{\overline{X}}

\font\cyr=wncyr10


\title[On quadratic Novikov algebras]{On quadratic Novikov algebras}

\author{Xiaofeng Dong}
\address{School of Mathematics, Hangzhou Normal University,
	Hangzhou, 311121, China}
\email{xfdong@stu.hznu.edu.cn}

\author{Yanyong Hong}
\address{School of Mathematics, Hangzhou Normal University,
	Hangzhou, 311121, China}
\email{yyhong@hznu.edu.cn}

\subjclass[2010]{17A45, 17A60, 17D25}

\keywords{Novikov algebra, quadratic Novikov algebra, double extension, quasi-Frobenius Novikov algebra}

\begin{abstract}
	A quadratic Novikov algebra is a Novikov algebra $(A, \circ)$ with a symmetric and nondegenerate bilinear form $B(\cdot,\cdot)$ satisfying $B(a\circ b, c)=-B(b, a\circ c+c\circ a)$ for all $a$, $b$, $c\in A$. This notion appeared in the theory of Novikov bialgebras. In this paper, we first investigate some properties of quadratic Novikov algebras and give a decomposition theorem of quadratic Novikov algebras. Then we present a classification of quadratic Novikov algebras of dimensions $2$ and $3$ over $\mathbb{C}$ up to isomorphism. Finally, a construction of quadratic Novikov algebras called double extension is presented and we show that any quadratic Novikov algebra containing a nonzero isotropic ideal can be obtained by double extensions. Based on double extension, an example of quadratic Novikov algebras of dimension 4 is given.
\end{abstract}

\maketitle
\delete{
	\vspace{-1.2cm}
	
	\tableofcontents
	
	\vspace{-1.2cm}}

\allowdisplaybreaks

\section{Introduction}
\vspace{-.2cm}
Novikov algebras were firstly introduced in the study of Hamiltonian
operators in the formal variational calculus~\mcite{GD1, GD2}. They also appeared in connection with Poisson brackets of hydrodynamic type~\mcite{BN}. It was shown in \cite{X1} that Novikov algebras correspond to a class of Lie conformal algebras which give an axiomatic
description of the singular part of operator product expansion of chiral fields in conformal field
theory \mcite{K1}. Moreover, Novikov algebras are also an important subclass of pre-Lie algebras,
which are closely related to many fields in mathematics and physics such as  affine manifolds and affine
structures on Lie groups \mcite{Ko},  convex homogeneous cones \mcite{V}, vertex algebras \mcite{BLP, BK} and deformation of
associative algebras \mcite{Ger}.

As we know, Killing form play important roles in the theory of Lie algebras. Therefore, it is also natural and important to investigate Novikov algebras with some bilinear form. Note that Novikov algebras with some bilinear form have been widely studied in \cite{AB, BM2, G, Le, SCM, Z, ZC}. In this paper, we plan to investigate a class of Novikov algebras with a symmetric and nondegenerate bilinear form satisfying a new invariant property, which are called quadratic Novikov algebras in \cite{HBG}. It was shown in \cite{HBG} that a finite-dimensional Novikov bialgebra is  equivalent to a Manin triple of Novikov algebras which is a quadratic Novikov algebra, and there is a Lie bialgebra structure on the tensor product of a finite-dimensional Novikov bialgebra and a finite-dimensional quadratic right Novikov algebra. In addition, by \cite{HBG1}, quadratic Novikov algebras have close relationships with Lie conformal algebras with a symmetric invariant nondegenerate conformal bilinear form. Therefore, quadratic Novikov algebras are important in the theories of Novikov bialgebras, Lie bialgebras and Lie conformal bialgebras. In addition, for a finite-dimensional quadratic Novikov algebra $(A, \circ, B(\cdot,\cdot))$, the adjoint representation of $(A, \circ)$ is isomorphic to the dual representation of the adjoint representation (see Proposition \ref{pro-1}). Therefore, quadratic Novikov algebras play a role in the theory of Novikov algebras analogous to the roles of Frobenius algebras in the theory of associative algebras and quadratic Lie algebras in the theory of Lie algebras. So it is meaningful to investigate quadratic Novikov algebras.

In this paper, we first investigate some properties of quadratic Novikov algebras. We show that for a finite-dimensional quadratic Novikov algebra $(A, \circ, B(\cdot,\cdot))$, the adjoint representation of $(A, \circ)$ is isomorphic to the dual representation of the adjoint representation. Moreover, by quadratic Novikov algebras, we present two constructions of quasi-Frobenius Novikov algebras which have close relationships with infinite-dimensional quasi-Frobenius Lie algebras \cite{HBG}. In addtion, we show that any finite-dimensional quadratic Novikov algebra can be decomposed into a direct sum of several nondegenerate ideals, where any two different nondegenerate ideals are orthogonal. Then a classification of quadratic Novikov algebras of dimensions $2$ and $3$ over $\mathbb{C}$ up to isomorphism is given. Note that double extensions appeared in the study of kinds of algebras with a bilinear form, see, e.g., \cite{AB, BB, BS, MR, WZ}. Motivated by these works, we introduce the notion of double extensions of a quadratic Novikov algebra $(A_1, \circ, B_1(\cdot,\cdot))$ by a Novikov algebra $(A_2, \circ)$ with a symmetric bilinear form $\tau(\cdot,\cdot)$.  It should be pointed out that $\tau(\cdot,\cdot)$  may not be invariant in the definition of double extensions, which is a little different from those in the study of other algebras with a bilinear form. The sufficient and necessary condition for such double extension is presented. We show that if a quadratic Novikov algebra has a nonzero isotropic ideal, then it is isomorphic to some double extension of a lower-dimensional quadratic Novikov algebra by a Novikov algebra with a symmetric bilinear form. By the definition of double extensions, it is shown that all nontrivial quadratic Novikov algebras of dimension $3$ are obtained from double extensions of a trivial quadratic Novikov algebra of dimension $1$ by some $1$-dimensional Novikov algebra with a symmetric bilinear form. Moreover, an example of $4$-dimensional quadratic Novikov algebras is given by double extensions.

This paper is organized as follows. In Section 2, we recall the definition of quadratic Novikov algebras and investigate some properties of quadratic Novikov algebras. Moreover, a decomposition theorem of quadratic Novikov algebras is given. Section 3 is devoted to classifying  quadratic Novikov algebras of dimensions $2$ and $3$ over $\mathbb{C}$ up to isomorphism. In Section 4,
the notion of double extensions of quadratic Novikov algebras is introduced and the sufficient and necessary condition for such double extension is presented. We show that if a quadratic Novikov algebra has a nonzero isotropic ideal, then it is isomorphic to some double extension of a lower-dimensional quadratic Novikov algebra by a Novikov algebra with a symmetric bilinear form. Moreover, applying double extensions, we present an example of $4$-dimensional quadratic Novikov algebras and show that all nontrivial quadratic Novikov algebras of dimension $3$ are obtained from double extensions of a trivial quadratic Novikov algebra of dimension $1$ by some $1$-dimensional Novikov algebra with a symmetric bilinear form.

\smallskip
\noindent
{\bf Notations.}
Throughout this paper, let  $\bf k$ be a field of characteristic zero and $\mathbb{C}$ be the field of complex numbers. All vector spaces and algebras over $\bf k$ are assumed to be finite-dimensional even though many results still hold in the infinite-dimensional cases. We denote the identity map by $\id$. Let $A$ be a vector space with a binary operation $\ast$.
Define linear maps
$L_{\ast}, R_{\ast}: A\rightarrow {\rm End}_{\bf k}(A)$ by
\begin{eqnarray*}
L_{\ast}(a)b:=a\ast b,\;\; R_{\ast}(a)b:=b\ast a, \;\;\;a, b\in A.
\end{eqnarray*}
\section{Basic results of quadratic Novikov algebras}
\mlabel{sec:dou}
\vspace{-.2cm}
In this section, we will introduce some basic results about quadratic Novikov algebras.
\subsection{Quadratic Novikov algebras}
\vspace{-.3cm}
\begin{defi}\cite{GD1}
	A {\bf Novikov algebra} $(A,\circ)$ is a vector space $A$ with a binary operation $\circ$ satisfying
	\begin{align*}
		(a \circ b) \circ c - a \circ (b \circ c) &= (b \circ a) \circ c - b \circ (a \circ c),\\
		(a \circ b)\circ c &= (a\circ c)\circ b	, \;\;\; a,b,c \in A.
	\end{align*}
	
	If $a\circ b = 0$ for all $a$, $b \in A$, we call $(A,\circ)$ a {\bf trivial Novikov algebra}.
	
	A {\bf subalgebra} $S$ (resp. an {\bf ideal} $I$) of a Novikov algebra $(A,\circ)$ is a vector subspace of $A$ satisfying $S\circ S \subset S$ (resp. $ I\circ A \subset I $ and $A\circ I \subset I$).
\end{defi}

For a Novikov algebra $(A, \circ)$, define another binary operation $\star$ on $A$ by
\begin{align*}
	a\star b: = a\circ b + b\circ a , \;\; a,b \in A.
\end{align*}

\delete{Next, we recall the definitions of solvable Novikov algebras and right nilpotent Novikov algebras.
\begin{defi} (see \cite{SZ} )
	Let $(A,\circ)$ be a Novikov algebra. We define a derived series $(A^{(n)})_{n\ge 0}: \; A^{(0)} = A,\; A^{(n+1)} = A^{(n)}A^{(n)} $ and a central descending series $(A^{n})_{n\ge 0}: \; A^{0} =A, \; A^{n+1} = A^{n}A$. Then $A$ is called {\bf solvable} (resp. {\bf right nilpotent}) of length $k$ if there is positive integer $k$ such that $A^{(k)} = 0$ (resp. $A^{k} = 0$).
\end{defi}}

For a bilinear form $B(\cdot ,\cdot ): A\times A \rightarrow {\bf k} $ on a vector space $A$,  $B(\cdot,\cdot )$ is called {\bf symmetric} (resp. {\bf skewsymmetric}) if $ B(a,b)=B(b,a)$ (resp. $B(a,b)= - B(b,a)$)  for all $a$, $b \in A$. Moreover, a symmetric (resp. skewsymmetric) $ B(\cdot,\cdot)$ is called {\bf nondegenerate} if $A^{\perp} = 0$, where $A^{\perp} = \{ a \in A \mid B(a,b) = 0 \;\text{for all}\; b \in A \}$.	

Next, we introduce the definition of quadratic Novikov algebras.
\begin{defi}\cite{HBG}
	A {\bf quadratic Novikov algebra} $(A,\circ,B(\cdot ,\cdot ))$ is a Novikov algebra $(A,\circ)$ with a nondegenerate symmetric bilinear form $B( \cdot,\cdot )$ satisfying the invariant property
	\begin{eqnarray}\label{def-eq1}
		B(a\circ b, c)=-B(b, a\star c),\;\;\;a, b, c\in A.
	\end{eqnarray}
\end{defi}

\begin{rmk}\label{rmk-1}
	For a quadratic Novikov algebra $(A,\circ,B(\cdot ,\cdot ))$, by Eq. (\ref{def-eq1}), we have $B(a\circ a ,a) = 0$ for all $a\in A$.
\end{rmk}

\begin{rmk}\label{metric-M}
	Let $\{e_{1},..., e_{n}\}$ be a basis of $A$. The bilinear form $B(\cdot ,\cdot )$ on $A$ under the basis is determined by the  metric matrix $B=(b_{ij})$, where $b_{ij} = B(e_{i},e_{j})$.
	$B(\cdot ,\cdot )$ is symmetric if and only if $B$ is symmetric and $B(\cdot ,\cdot )$ is nondegenerate if and only if the determinant of $B$ is not zero.
\end{rmk}

\begin{defi}\cite{O}
	A {\bf representation} of a Novikov algebra $(A,\circ)$ is a triple $(V,l,r)$, where $V$ is a vector space and $l, r: A\rightarrow \text{End}_{\bf k}(V)$ are linear maps satisfying
	\begin{align*}
		l(a\circ b - b\circ a)v &= l(a)l(b)v - l(b)l(a)v,\\
		l(a)r(b)v - r(b)l(a)v &= r(a\circ b)v - r(a)r(b)v,\\
		l(a\circ b)v &= r(b)l(a)v,\\
		r(a)r(b)v &=  r(b)r(a)v, \;\; a,b \in A,\;\; v\in V.
	\end{align*}
\end{defi}

For a linear map $\varphi:A\rightarrow \text{End}_{{\bf k}}(V)$, we define a linear map $\varphi^{\ast}:A\rightarrow \text{End}_{{\bf k}}(V^{\ast})$ by
\begin{align*}
	\langle \varphi^{\ast}(a)f,v \rangle = -\langle f,\varphi(a)v \rangle, \;  a\in A,\; f\in V^{\ast},\; v\in V.
\end{align*}

\begin{rmk}Note that $(A,L_{\circ},R_{\circ})$ is a representation of $(A,\circ)$, \delete{where
	\begin{align*}
		L_{\circ}(a)(b) \coloneqq a\circ b,\;\; R_{\circ}(a)(b) \coloneqq b\circ a, \;\; a,b \in A.
	\end{align*}}
which is called the {\bf adjoint representation} of $(A,\circ)$. Moreover, by \cite[Proposition 3.3]{HBG}, $(A^{\ast},L^{\ast}_{\star},-R^{\ast}_{\circ})$ is also a representation of $(A,\circ)$, where $L^{\ast}_{\star} = L^{\ast}_\circ + R^{\ast}_\circ$.
\end{rmk}

\begin{defi}
	Let $(V_{1},l_{1},r_{1})$ and $(V_{2},l_{2},r_{2})$ be two representations of a Novikov algebra $(A,\circ)$. If a linear map $\varphi: V_1\rightarrow V_2$ satisfies
	\begin{align*}
		\varphi(l_{1}(a)v_{1}) &= l_{2}(a)\varphi(v_{1}),\\
		\varphi(r_{1}(a)v_{1}) &= r_{2}(a)\varphi(v_{1}), \;\; v_{1}\in V_{1},\;\;a\in A,
	\end{align*}
then we call that $\varphi$ is a {\bf homomorphism of representations} from $(V_{1},l_{1},r_{1})$ to $(V_{2},l_{2},r_{2}) $. Moreover, if $\varphi$ is a bijection, then we say that $\varphi$ is an {\bf isomorphism of representations} from $(V_{1},l_{1},r_{1})$  to  $(V_{2},l_{2},r_{2}) $.
\end{defi}

\begin{pro}\label{pro-1}
	Let  $(A,\circ,B( \cdot,\cdot ))$ be a quadratic Novikov algebra. Then the adjoint representation $(A,L_{\circ},R_{\circ})$ of $(A, \circ)$ is isomorphic to the  representation $(A^{\ast},L^{\ast}_{\star},-R^{\ast}_{\circ})$.
\end{pro}
\begin{proof}
	Let $\theta: A\rightarrow A^{\ast}$ be the linear map defined by $\theta(a) = B(a,\cdot)$ for all $a\in A$. Since $B( \cdot,\cdot )$ is nondegenerate, $\theta$ is a bijection. For all $a$, $b$, $c\in A$, we have
	\begin{align*}
		\theta(L_{\circ}(a)b)(c) &= \theta(a\circ b)(c) = B(a\circ b,c) = -B(b,a\star c) \\&= -B(b,L_{\star}(a)c) =  -\langle \theta(b), L_{\star}(a)c \rangle \\&=  (L^{\ast}_{\star}(a)\theta(b))(c),
\end{align*}
and
\begin{align*}
		\theta(R_{\circ}(a)b)(c) &= \theta(b\circ a)(c) = B(b\circ a,c) = B(c\circ a,b) \\&= B(b,R_{\circ}(a)c) = \langle \theta(b), R_{\circ}(a)c \rangle \\ &= (-R^{\ast}_{\circ}(a)\theta(b))(c).
	\end{align*}
	Therefore, $\theta$ is an isomorphism of representations from $(A,L_{\circ},R_{\circ})$ to $(A^{\ast},L^{\ast}_{\star},-R^{\ast}_{\circ})$.
\end{proof}

\delete{\begin{defi}
		A {\bf derivation} d of Novikov algebra $(A,\circ)$ is a linear map $d\in End(A)$ satisfying
		\begin{align*}
			d(a\circ b) = d(a)\circ b + a\circ d(b), \;\; \text{for all} \;\; a,b\in A.
		\end{align*}
\end{defi}}

Finally, we give two constructions of quasi-Frobenius Novikov algebras from quadratic Novikov algebras.
\begin{defi}\cite{HBG}
	A {\bf quasi-Frobenius Novikov algebra} $(A,\circ,\omega ( \cdot,\cdot ))$ is a Novikov algebra $(A,\circ)$ with a nondegenerate skewsymmetric bilinear form $\omega ( \cdot,\cdot )$ satisfying
	\begin{eqnarray}\label{eq:qF}
		\omega(a\circ b,c) - \omega(a\star c,b) + \omega(c\circ b,a) = 0,\;\; a,b,c \in A.
	\end{eqnarray}
\end{defi}

\begin{pro}\label{pro-qN-1}
	Let $(A,\circ,B( \cdot,\cdot ))$ be a quadratic Novikov algebra and $D$ be a derivation of $(A, \circ)$ such that $B(D(a),b) = -B(a,D(b))$ for all $a$, $b\in A$. Define $\omega(\cdot,\cdot): A\times A\rightarrow {\bf k}$ by
	\begin{eqnarray}
		\omega (a,b) = B(D(a),b),\;\;a, b\in A.
	\end{eqnarray} Then $\omega(\cdot,\cdot)$ is a skewsymmetric bilinear form satisfying Eq. (\ref{eq:qF}).
\end{pro}
\begin{proof}
	Since
	\begin{align*}
		\omega (a,b) = B(D(a),b) = -B(a,D(b)) = -B(D(b),a) = -\omega (b,a),\;\;a, b\in A,
	\end{align*}
	$\omega ( \cdot,\cdot )$ is skewsymmetric.
	
	Let $a$, $b$, $c\in A$. Note that
	\begin{align*}
		\omega (a\circ b,c) =& B(D(a\circ b),c) = B(D(a)\circ b + a\circ D(b),c) \\=& -B(b,D(a)\star c) - B(D(b),a\star c) = -B(b,D(a)\star c) + B(b,D(a\star c)) \\ =& B(b,a\star D(c)),\\
		\omega (c\circ b,a) =& B(D(c\circ b),a) = B(D(c)\circ b + c\circ D(b),a) \\=& -B(b,D(c)\star a) - B(D(b),c\star a) = -B(b,D(c)\star a) + B(b,D(c\star a)) \\ =& B(b,c\star D(a)),\\
		\omega (a\star c,b) =& B(D(a\star c),b) = B(D(a)\star c + a\star D(c),b).
	\end{align*}
	Therefore, Eq. (\ref{eq:qF}) holds.
\end{proof}

\begin{rmk}
	If $D$ is a bijection in Proposition \ref{pro-qN-1}, then $(A,\circ,\omega ( \cdot,\cdot ))$ is a quasi-Frobenius Novikov algebra.
\end{rmk}

\begin{pro}
	Let $(A,\circ,B( \cdot,\cdot ))$ be a quadratic Novikov algebra and $D: A\rightarrow A$ be a linear isomorphism such that $D(a\circ b) = \frac{1}{2}a\circ D(b)$ and $B(D(a),b) = -B(a,D(b))$ for all $a$, $b\in A$. Then $(A,\circ,\omega ( \cdot,\cdot ))$ is a quasi-Frobenius Novikov algebra where $\omega ( \cdot,\cdot )$ is defined by $\omega (a,b) = B(D(a),b)$ for all $a,b \in A$.
\end{pro}
\begin{proof}
	Since $D$ is a bijection and $B(\cdot,\cdot)$ is nondegenerate, we have that $\omega(\cdot,\cdot)$ is nondegenerate. By Proposition \ref{pro-qN-1},
	$\omega ( \cdot,\cdot )$ is skewsymmetric.
	
	For all $a$, $b$, $c\in A$, we have
	\begin{align*}
		\omega (a\circ b,c) =& B(D(a\circ b),c) = \frac{1}{2}B(a\circ D(b),c) \\=& -\frac{1}{2}B(D(b),a\star c) = \frac{1}{2}B(b,D(a\star c)),\\
		\omega (c\circ b,a) =& B(D(c\circ b),a) = \frac{1}{2}B( c\circ D(b),a) \\=& -\frac{1}{2}B(D(b),c\star a) = \frac{1}{2}B(b,D(c\star a)),\\
		\omega (a\star c,b) =& B(D(a\star c),b).
	\end{align*}
	Then  $\omega(\cdot,\cdot)$ satisfies Eq. (\ref{eq:qF}).
	Therefore, $(A,\circ,\omega ( \cdot,\cdot ))$ is a quasi-Frobenius Novikov algebra.
\end{proof}

\subsection{Structures of quadratic Novikov algebras}
First, we introduce some basic definitions about quadratic Novikov algebras.
\begin{defi}
Let $(A, \circ, B(\cdot,\cdot)$ be a quadratic Novikov algebra.

A {\bf subalgebra} $S$ (resp. an {\bf ideal} $I$) of $(A,\circ,B(\cdot ,\cdot))$ is just the subalgebra (resp. ideal) of $(A,\circ)$.
	
	A subalgebra (resp. an ideal) $S$ of $(A, \circ, B(\cdot,\cdot))$ is called {\bf nondegenerate} if the restriction of $B(\cdot ,\cdot )$ on $S$ is a nondegenerate bilinear form \delete{and $S\neq 0$}.
	
	A subalgebra (resp. an ideal) $S$ of $(A, \circ, B(\cdot,\cdot))$ is called {\bf isotropic} if $B(S,S) = 0$.
\end{defi}

\begin{lem}\label{nonis}
	Let $(A,\circ,B( \cdot,\cdot ))$ be a quadratic Novikov algebra. Then the following conclusions hold.
	
	(1) If $I$ is an ideal of $(A, \circ)$, then $I^{\perp}=\{a\in A\mid B(a, b)=0 \;\text{for all $b\in I$}\}$ is an ideal of $(A, \circ)$.
	
	(2) If $S$ is a nondegenerate subalgebra of $(A,\circ,B( \cdot,\cdot ))$, then $S^{\perp}$ is also a nondegenerate subalgebra of $(A,\circ,B( \cdot,\cdot ))$ and $A = S\oplus S^{\perp}$ as vector spaces.
\end{lem}
\begin{proof}
	(1) Let $a \in A, b \in I^{\perp}$ and $c \in I$. Since $c \in I$, we have $a\star c = a \circ c+ c\circ a \in I$. Therefore, according to that $b \in I^{\perp}$, we have $B(a\circ b,c) = -B(b,a\star c) = 0$, which means that $a\circ b \in I^{\perp}$. Note that $B(a\circ c,b) = -B(c,a\circ b) - B(c,b\circ a)$. Since $a\circ c \in I$ and $a\circ b \in I^{\perp}$, we obtain $B(c,b \circ a) = 0$. Then we have $b \circ a \in I^{\perp}$. Therefore, $I^{\perp}$ is an ideal of $(A, \circ)$.
	
	(2) For any $a \in S\cap S^{\perp}$, we have $B(a,S) = 0$. Since $B(\cdot,\cdot)|_{S\times S}$ is nondegenerate, we get $a=0$. Therefore, the sum $S+S^{\perp}$ is the direct sum as vector spaces.
	
	Since $S\subset A$ and $S^{\perp}\subset A$, we have $S \oplus S^{\perp} \subset A$. Next we only prove that $A\subset S\oplus S^{\perp}$.
	
	If $S = A$, we have $A\subset S$. If $S = 0$, we get $B(S,A) = 0$, that is $A \subset S^{\perp}$. Therefore, in both two cases, $A \subset S\oplus S^{\perp}$.

	Consider the case that $0 \hspace{1pt} < \hspace{1pt} \text{dim}S= m < \hspace{2pt} \text{dim}A$. Let $\{x_{1},...,x_{m}\}$ be a basis of $S$. Suppose that $y \in A$, but $y \notin S$. Consider the system of equations with variables $\lambda_1,..., \lambda_{m+1}$ as follows.
	\begin{eqnarray}
		\sum_{j=1}^{m}B(x_{i},x_{j})\lambda_{j} + B(y,x_{i})\lambda_{m+1} = 0,\;\;\;i=1,\cdots, m.
	\end{eqnarray}
	Since this system of equations has $m+1$ variables and $m$ equations, there exists a nonzero solution $(\alpha_{1},\cdots,\alpha_{m+1})$, where $\alpha_i\in {\bf k}$.
	Let $z= \sum_{j=1}^{m}\alpha_{j}x_{j} + \alpha_{m+1}y\in A$. Then we have  $B(z,x_{i}) = 0$ for $i=1,..., m$. Thus we have $z \in S^{\perp}$. If $\alpha_{m+1} = 0$, we obtain $z \in S \cap S^{\perp}$. Thus, $z =0$. Since $\{x_{1},...,x_{m}\}$ is a basis of $S$, we have $\alpha_{1},...,\alpha_{m}$ are zero which contradicts to that $(\alpha_{1},...,\alpha_{m})$ is nonzero. Thus $\alpha_{m+1} \neq 0$. Then we get $y = \frac{1}{\alpha_{m+1}}z - \frac{1}{\alpha_{m+1}}\sum_{j=1}^{m}\alpha_{j}x_{j} \in S\oplus S^{\perp}$. Thus $A\subset S \oplus S^{\perp}$. Consequently, we have $A$ = $S$ $\oplus$ $S^{\perp}$.
	
	Suppose that $S^{\perp}$ is degenerate. Then there exists a nonzero element $x \in S^{\perp}$ such that $B(x,S^{\perp}) = 0$. Since $B(x,S) = 0$, we have $B(x,A) = 0$. By the nondegenerate property of $B(\cdot,\cdot)$, we obtain $x = 0$, which contradicts with $x\neq 0$. Therefore, $S^{\perp}$ is a nondegenerate subalgebra.
\end{proof}

\begin{rmk}
 By Lemma \ref{nonis}, if $I$ is a nondegenerate ideal of $(A,\circ,B( \cdot,\cdot ))$, then $I^{\perp}$ is also a nondegenerate ideal of $(A,\circ,B( \cdot,\cdot ))$ and $A = I\oplus I^{\perp}$ as vector spaces.
\end{rmk}

Finally, we present a decomposition theorem of quadratic Novikov algebras.
\begin{thm}
	Let $(A,\circ,B(\cdot,\cdot))$ be a quadratic Novikov algebra. Then there exists a positive integer $r$ such that $A = \oplus_{i = 1}^r A_{i}$, where
	
	(1) $A_{i}$ is a nondegenerate ideal of $(A,\circ,B(\cdot,\cdot))$ for each $i$;
	
	(2) $A_{i}$ contains no nontrivial nondegenerate ideal of $(A,\circ,B(\cdot,\cdot))$ for each $i$;
	
	(3) for all $i\neq j$, $A_{i}$ and $A_{j}$ are orthogonal.
\end{thm}

\begin{proof}
	If $(A,\circ,B(\cdot,\cdot))$ contains no nontrivial nondegenerate ideal, then this conclusion naturally holds.
	
	Suppose that $(A,\circ,B(\cdot,\cdot))$ contains nontrivial nondegenerate ideals. Then by Lemma \ref{nonis}, there exists a nontrivial nondegenerate ideal $I$ of $(A,\circ,B(\cdot,\cdot))$  such that $A = I\oplus I^{\perp}$, where $I^{\perp}$ is also a nontrivial nondegenerate ideal of  $(A,\circ,B(\cdot,\cdot))$. If both $I$ and $I^{\perp}$ contain no nontrivial nondegenerate ideal of $(A,\circ,B(\cdot,\cdot))$, then this conclusion holds. Then we consider the case that either $I$ or $I^{\perp}$ contains nontrivial nondegenerate ideals of $(A,\circ,B(\cdot,\cdot))$. Assume that $J$ is a nontrivial nondegenerate ideal of $(A,\circ,B(\cdot,\cdot))$ in $I$. Then $J$ is a nontrivial nondegenerate ideal of $(I, \circ, B(\cdot,\cdot)|_{I\times I})$. Since $B(\cdot,\cdot)|_{I\times I}$ is nondegenerate, by Lemma \ref{nonis},  we have $I=J\oplus J^{\perp}$, where $J^{\perp}$ is also a nondegenerate ideal of  $(I, \circ, B(\cdot,\cdot)|_{I\times I})$. Then we have $A = J\oplus J^{\perp}\oplus I^{\perp}$.
	By Lemma \ref{nonis}, we have $J\circ J^{\perp}$, $J^{\perp}\circ J \subset J\cap J^{\perp} = 0$, and $I\circ I^{\perp}$, $I^{\perp}\circ I \subset I\cap I^{\perp} = 0$. Thus, we obtain $J\circ J^{\perp}=0$, $J^{\perp}\circ J=0$, $I\circ I^{\perp}=0$ and $I^{\perp}\circ I=0$. Therefore, $J^{\perp}\circ A=J^{\perp}\circ ( J\oplus J^{\perp}\oplus I^{\perp})=J^{\perp}\circ J^{\perp}\subset J^{\perp}$. Similarly, we get $A\circ J^{\perp}\subset J^{\perp}$. Thus, $J^{\perp}$ is an ideal of $(A, \circ)$, i.e. $J^{\perp}$ is a nondegenerate ideal of  $(A, \circ, B(\cdot,\cdot))$.
	
	Because $A$ is finite-dimensional, by finite steps, we will end up with a decomposition of $A$ satisfying (1)-(3).
	This completes the proof.
\end{proof}
	
\section{Classification of low dimensional quadratic Novikov algebras over $\mathbb{C}$}
In this section, let ${\bf k}$ be $\mathbb{C}$. We will present a classification of quadratic Novikov algebras of dimensions $2$ and $3$  over $\mathbb{C}$ up to isomorphism.

Let $\{e_1,\ldots, e_n\}$ be a basis of a Novikov algebra $(A, \circ)$. Set $e_i\circ e_j=\sum_{k=1}^nc_{ij}^ke_k$ for each $i$, $j\in \{1, \ldots, n\}$. Then $(A, \circ)$ is determined by the {\bf characteristic matrix}
\begin{eqnarray}
\mathcal{A}=\left(
               \begin{array}{ccc}
                 \sum_{k=1}^n c_{11}^k e_k & \cdots & \sum_{k=1}^n c_{1n}^k \\
                 \cdots & \cdots & \cdots \\
                  \sum_{k=1}^n c_{n1}^k e_k & \cdots &  \sum_{k=1}^n c_{nn}^k e_k \\
               \end{array}
             \right).
\end{eqnarray}
According to Remark \ref{metric-M}, a quadratic Novikov algebra $(A, \circ, B(\cdot, \cdot))$ is determined by the characteristic matrix of $(A, \circ)$ and the metric matrix $B=(b_{ij})$, where $b_{ij} = B(e_{i},e_{j})$.

Next, we introduce the definition of isomorphisms of quadratic Novikov algebras.
\begin{defi}
	Let $(A_1,\circ,B_1(\cdot,\cdot))$ and $(A_2,\circ,B_2(\cdot,\cdot))$ be two quadratic Novikov algebras. If a linear map $\varphi: A_1\rightarrow A_2$ satisfies
	\begin{align*}
		\varphi(a\circ b) &= \varphi(a)\circ \varphi(b),\\
		B_1(a,b) &= B_2(\varphi(a),\varphi(b)), \;\; a,b \in A_1,
	\end{align*}
then we call that $\varphi$ is a {\bf homomorphism} from $(A_1,\circ,B_1(\cdot,\cdot))$ to $(B_2,\circ, B_2(\cdot,\cdot))$.
	Moreover, $\varphi$ is called {\bf an isomorphism of quadratic Novikov algebras} if $\varphi$ is a bijection.
\end{defi}

\begin{rmk}\label{rmk-class}
Let $(A_1, \circ, B_1(\cdot,\cdot))$ be a quadratic Novikov algebra and $(A_2, \circ)$ be a Novikov algebra. If $\varphi: A_1\rightarrow A_2$ is an isomorphism of Novikov algebras, then $(A_2, \circ, B_1(\cdot, \cdot)(\varphi^{-1}\times \varphi^{-1}))$ is a quadratic Novikov algebra which is isomorphic to $(A_1, \circ, B_1(\cdot,\cdot))$. Based on this, there are two steps to classify quadratic Novikov algebras of dimension $n$ as follows.
\begin{enumerate}
\item\label{itm1} Classify Novikov algebras of dimension $n$ up to isomorphism.
\item Characterize all symmetric invariant nondegenerate bilinear forms $B(\cdot,\cdot)$ on the obtained Novikov algebras in Step (\ref{itm1}).
\end{enumerate}
\end{rmk}

\begin{lem}
	\label{not1}
	There does not exist any quadratic Novikov algebra structure on a Novikov algebra with an identity.
\end{lem}
\begin{proof}
	Let $(A, \circ)$ be a Novikov algebra with an identity $e$ and  $B(\cdot,\cdot)$ be an invariant bilinear form on $A$. For all $a\in A$, we have
\begin{eqnarray*}
B(e, a)=B(e\circ e, a)=-B(e, e\star a)=-2B(e, a).
\end{eqnarray*}
Therefore, we get $B(e, a)=0$ for all $a\in A$, i.e. $B(\cdot,\cdot)$ is degenerate. Thus this lemma holds.
\end{proof}
\begin{rmk}
By Lemma \ref{not1}, there are no nontrivial quadratic Novikov algebras of dimension $1$.
\end{rmk}

Next, we will classify quadratic Novikov algebras of dimensions $2$ and $3$ over $\mathbb{C}$ up to isomorphism. Note that
Novikov algebras of dimensions $2$ and $3$  over $\mathbb{C}$ up to isomorphism have been classified in \cite{BM1}.
By Remark \ref{rmk-class}, for classifying quadratic Novikov algebras of dimensions $2$ and $3$ up to isomorphism, we only need to consider those symmetric invariant nondegenerate bilinear forms on those Novikov algebras classified in \cite{BM1}.

Note that by the result in \cite{BM1}, the classification of Novikov algebras of dimension $2$ up to isomorphism is given by the following table.
	\begin{table}[H]
		\centering
		\caption{Classification of $2$-dimensional  Novikov algebras up to isomorphism}
		\begin{tabular}{c|c|c|c}
			Type & characteristic matrix & Type & characteristic matrix\\ \hline
			(T1) & $\begin{pmatrix}
				0 & 0\\
				0 & 0
			\end{pmatrix}$ &
			(T2) & $\begin{pmatrix}
				e_{2} & 0\\
				0 & 0
			\end{pmatrix}$\\ \hline
			(T3) & $\begin{pmatrix}
				0 & 0\\
				-e_{1} & 0
			\end{pmatrix}$ &
			(N1) & $\begin{pmatrix}
				e_{1} & 0\\
				0 & e_{2}
			\end{pmatrix}$\\ \hline
			(N2) & $\begin{pmatrix}
				e_{1} & 0\\
				0 & 0
			\end{pmatrix}$ &
			(N3) & $\begin{pmatrix}
				e_{1} & e_{2}\\
				e_{2} & 0
			\end{pmatrix}$\\ \hline
			(N4) & $\begin{pmatrix}
				0 & e_{1}\\
				0 & e_{2}
			\end{pmatrix}$ &
			(N5) & $\begin{pmatrix}
				0 & e_{1}\\
				0 & e_{2} + e_{1}
			\end{pmatrix}$\\ \hline
			(N6) & $\begin{pmatrix}
				0 & e_{1}\\
				le_{1} & e_{2}
			\end{pmatrix}$,  $l \neq 0,1$ 		
		\end{tabular}
	\end{table}
\begin{lem}
	\label{not2}
	Let $(A,\circ)$ be a nontrivial $2$-dimensional Novikov algebra with a nonzero element $e_{1}$ such that $a\circ e_{1} = 0$ for all $a\in A$. Then there does not exist any quadratic Novikov algebra structure on $(A,\circ)$.
\end{lem}
\begin{proof}
	Suppose that $(A, \circ, B(\cdot,\cdot))$ is a quadratic Novikov algebra whose basis is $\{e_{1},e_{2}\}$. Then we have
	\begin{eqnarray*}
		&&B(e_{1}\circ e_{2},e_{1}) =-B(e_{2},e_{1}\star e_{1}) = 0,\\
		&&B(e_{1}\circ e_{2},e_{2}) =-B(e_{2},e_{1}\star e_{2}) =-B(e_{2},e_{1}\circ e_{2}).
	\end{eqnarray*}
Therefore, $B(e_1\circ e_2, e_1)=B(e_1\circ e_2, e_2)=0$. Since $B(\cdot,\cdot)$ is nondegenerate, we get $e_1\circ e_2=0$.
Thus we obtain
	\begin{eqnarray*}
		&&B(e_{2}\circ e_{2},e_{1}) =-B(e_{2},e_{2}\star e_{1}) = 0.
\end{eqnarray*}
By Remark \ref{rmk-1}, one gets $B(e_{2}\circ e_{2},e_{2}) = 0$.
By the nondegenerate property of $B(\cdot,\cdot)$, we get $e_{2}\circ e_{2} = 0$. Thus $(A,\circ)$ is a trivial Novikov algebra. Therefore, this conclusion holds.
\end{proof}

\begin{thm}
	\label{2dim}	
	Any $2$-dimensional nontrivial quadratic Novikov algebra is isomorphic to the quadratic Novikov algebra $(A=\mathbb{C}e_1\oplus \mathbb{C}e_2,\circ, B(\cdot,\cdot))$ which is defined by
\begin{eqnarray}
&&e_1\circ e_1=0,\;\;e_1\circ e_2=e_1,\;\;e_2\circ e_1=-2e_1,\;\;e_2\circ e_2=e_2,\\
&&B(e_1,e_1)=k,\;\;B(e_1, e_2)=B(e_2, e_1)=s,\;\;B(e_2,e_2)=0,\label{dim-2}
\end{eqnarray}
for some $k\in \mathbb{C}$ and $s\in \mathbb{C}\backslash \{0\}$.
\end{thm}
\begin{proof}
	By Remark \ref{rmk-class}, we can assume that $(A, \circ)$ is one of the cases in Table 1.
	
	For Case (T1) in Table 1, $(A,\circ)$ is the trivial Novikov algebra. By Lemma \ref{not1}, there does not exist any quadratic Novikov algebra structure on  Case (N3). By Lemma \ref{not2}, there are no quadratic Novikov algebra structures on Cases (T2), (T3), (N2), (N4) and (N5).
	
	Suppose that there exists a nonzero element $e_{i}\in A$ such that $e_{i}\circ e_{i} = e_{i}$. Without loss of generality, we assume that $i = 2$ and $e_{2} \circ e_{1} = pe_{1} + qe_{2}$ for some $p$, $q\in \mathbb{C}$. Then we have
	\begin{align*}
		B(e_{2} \circ e_{1},e_{2}) =& B(pe_{1} + qe_{2},e_{2}) = pB(e_{1},e_{2}) + qB(e_{2},e_{2}) \\=&
		-B(e_{1},e_{2} \star e_{2}) = -2B(e_{1},e_{2}).
	\end{align*}
	Note that $B(e_{2}, e_{2})=0$. Therefore,  $(p + 2)B(e_{1},e_{2}) = 0$. If $p + 2 \neq 0$, then $B(e_{1},e_{2}) = 0$.
Therefore, the bilinear form $B(\cdot ,\cdot)$ is degenerate in Cases (N1) and (N6) when $l\neq -2$.

Finally, we only need to consider Case (N6) when $l=-2$. It is straightforward to compute that any symmetric invariant nondegenerate bilinear form on this case is of the form defined by Eq. (\ref{dim-2}). Then the proof is completed.
	
\delete{	If $p + 2 = 0$ then $B(e_{1},e_{2})$ may be not $0$. So if $p = -2$ and $B(e_{1},e_{2}) \neq 0$, the determinant is $-(B(e_{1},e_{2}))^2$ and the bilinear form $B(\cdot ,\cdot)$ is nondegenerate in the case (N6) where $l = -2$.
	So in the case (N6) and $l = -2$, $A$ could become a quadratic Novikov algebra with the bilinear form $B(\cdot ,\cdot)$,
	
	Let $f: A\rightarrow N$ be a linear map such that
	\begin{align*}
		B(e_{1},e_{2}) = \bar{B}(f(e_{1}),f(e_{1})),\\
		B(e_{1},e_{2}) = \bar{B}(f(e_{1}),f(e_{2})),\\
		B(e_{2},e_{2}) = \bar{B}(f(e_{1}),f(e_{1})).
	\end{align*}
	Hence $f$ is an isomorphism of quadratic Novikov algebras. This completes the proof.}
\end{proof}
\begin{thm}\label{3-dim}
Any $3$-dimensional nontrivial quadratic Novikov algebra is isomorphic to $(A=\mathbb{C}e_1\oplus \mathbb{C}e_2\oplus \mathbb{C}e_3, \circ, B(\cdot,\cdot))$ which is one of cases in the following table.
\begin{table}[H]
	  	\centering
	  	\caption{Classification of $3$-dimensional nontrivial quadratic Novikov algebras up to isomorphism}
	  	\begin{tabular}{c|c|c}
	  		\hline
	  		Type & Characteristic matrix of $(A,\circ)$ & metric matrix of the bilinear form\\ \hline
	  		(A7)($l= -2$) & $\begin{pmatrix}
	  			0 & 0 & 0\\
	  			0 & 0 & e_{1}\\
	  			0 & -2e_{1} & e_{2}
	  		\end{pmatrix}$ &
	  		$\begin{pmatrix}
	  		0 & 0 & k\\
	  		0 & k & 0\\
	  		k & 0 & t
	  		\end{pmatrix},$ $k, t\in \mathbb{C}$, $k^3\neq 0$	\\
	  		\hline
	  		(C5)($l= -2$) & $\begin{pmatrix}
	  			0 & 0 & e_{1}\\
	  			0 & 0 & 0\\
	  			-2e_{1} & 0 & e_3
	  		\end{pmatrix}$ &
	  		$\begin{pmatrix}
	  			0 & 0 & k\\
	  			0 & s & 0\\
	  			k & 0 & 0
	  		\end{pmatrix}$, $k, s\in \mathbb{C}$, $k^2s\neq 0$ 	\\
	  		\hline
	  		(D6)($l= -\frac{1}{2}$) & $\begin{pmatrix}
	  			e_{2} & 0 & e_{1}\\
	  			0 & 0 & e_{2}\\
	  			-\frac{1}{2}e_{1} & -2e_{2} & e_{3}
	  		\end{pmatrix}$ &
	  		$\begin{pmatrix}
	  			-2s & 0 & 0\\
	  			0 & 0 & s\\
	  			0 & s & 0
	  		\end{pmatrix}$, $s\in \mathbb{C}$, $s^3\neq 0$  \\ \hline
	  	\end{tabular}
	  \end{table}

	\delete{Let $(A,\circ,B(\cdot,\cdot))$ be a nontrivial three-dimensional quadratic Novikov algebra equipped with a basis $\{e_{1},e_{2},e_{3}\}$. Then $A$ is isomorphic to quadratic Novikov $(N,\circ,\bar{B}(\cdot,\cdot))$ algebra whose algebraic structure is one of the cases $(A7)(l = -2),(A11)(l = -1),(C5)(l = -1) and (D6)(l = -\frac{1}{2})$ from \cite{BD1}.}
\end{thm}
\begin{proof}
Note that all $3$-dimensional Novikov algebras up to isomorphism have been classified in \cite{BM1}. By Remark \ref{rmk-class}, we can assume that $(A, \circ)$ is one of the cases classified in \cite{BM1}. Then we consider symmetric invariant nondegenerate bilinear forms on $(A, \circ)$
	
{\bf Case 1:} Assume that there exist $e_{i}$, $e_{k}\in \{e_1,e_2, e_3\}$ $(i\neq k)$ such that $e_{i}\circ e_{i} = e_{k}$ and $e_{i}\star e_{j} = 0$ for each $j\neq i$. Then we have \begin{align*}
		B(e_{i}\circ e_{i},e_{j}) &= B(e_{k},e_{j}) = -B(e_{i},e_{i}\star e_{j}) = 0.
	\end{align*}
Note that $B(e_k, e_i)=B(e_{i}\circ e_{i},e_{i})=0$. Therefore, the bilinear form $B(\cdot,\cdot)$ in this case is degenerate. Thus there are no quadratic Novikov algebra structures on Cases (A2), (A3) and (A6) given in  \cite[Table 2]{BM1}.

{\bf Case 2:} Suppose that the basis $\{e_{i},e_{j},e_{k}\}$ of $(A,\circ)$ satisfies $e_{i}\circ e_{j} = e_{k}$, $e_{i}\star e_{l} = 0$ for each $l\neq j$,  and $e_{j}\circ e_{i} = me_{k}$ $(m\neq -2)$.
Then we have
	 \begin{align*}
	 	B(e_{i}\circ e_{j},e_{l}) &= B(e_{k},e_{l}) = -B(e_{j},e_{i}\star e_{l}) = 0,\\
	 	B(e_{i}\circ e_{j},e_{j}) &= B(e_{k},e_{j}) = -B(e_{j},e_{i}\star e_{j}) \\=& -B(e_{j},(m+1)e_{k}) = 0.
	 \end{align*}
Hence the bilinear form $B(\cdot,\cdot)$ in this case is degenerate. Therefore there are no quadratic Novikov algebra structures on Cases (A4), (A5), (A7) ($l \neq -2$), (A9), (C4), (C5)($l \neq$ -2), (C18), (C19)  given in  \cite[Tables 2 and 4]{BM1}.
	
{\bf Case 3:} Assume that there exists $e_{i}\in \{e_1, e_2, e_3\}$ such that $e_{i}\circ e_{i} = e_{i}$ and $e_{i}\star e_{j} = me_{j}$ $(m\neq -1)$ for each $j\neq i$.
Then we have
	 \begin{align*}
	 	&B(e_{i}\circ e_{i},e_{j}) = B(e_{i},e_{j}) = -B(e_{i},e_{i}\star e_{j}) = -mB(e_{i},e_{j}) = 0.
	 \end{align*}
Note that $B(e_i,e_i)=B(e_i\circ e_i, e_i)=0$. Hence the bilinear form $B(\cdot,\cdot)$ in this case is degenerate. Therefore there are no quadratic Novikov algebra structures on Cases (B1)-(B5), (C1)-(C3), (D1), (D2), (D4), (E1) given in \cite[Tables 3, 4, 5, 6]{BM1}.
	
{\bf Case 4:} Assume that there exists $e_{j}\in \{e_1, e_2, e_3\}$  such that $e_{i}\circ e_{j} = e_{i}$ for each $i\neq j$ and $e_{k}\circ e_{l} = 0$ for each $k\neq j$ and $l\neq j$. Without loss of generality, we take $j = 3$. Then we have
	  \begin{align*}
	  	&B(e_{k}\circ e_{3},e_{l}) = B(e_{k},e_{l}) = -B(e_{3},e_{k}\star e_{l}) = 0.
	  \end{align*}
Hence the bilinear form $B(\cdot,\cdot)$ in this case is degenerate. Therefore there are no quadratic Novikov algebra structures on Cases (C6)-(C17) given in  \cite[Table 4]{BM1}.
	
{\bf Case 5}: Assume that the basis $\{e_1, e_{2},e_{3}\}$ of $(A, \circ)$ satisfies
\begin{eqnarray*}
e_{2}\circ e_{2} = 0, \;\;e_{2}\circ e_{3} = e_{2}, \;\;e_{3}\circ e_{2} = me_{2}\;(m\neq -2),\;\;e_{3}\circ e_{3} = ke_{2} + le_{3}\;(l\neq 0).
	  \end{eqnarray*}
	  Then we have
	  \begin{align*}
	  	&B(e_{2}\circ e_{3},e_{2}) = B(e_{2},e_{2}) = -B(e_{3},e_{2}\star e_{2}) = 0,\\
	  	&B(e_{2}\circ e_{3},e_{3}) = B(e_{2},e_{3}) = -B(e_{3},e_{2}\star e_{3}) = -B(e_{3},(m + 1)e_{2}) = 0.
\end{align*}
Note that $B(e_{3}\circ e_{3},e_{3}) = B(ke_{2} + le_{3},e_{3}) = lB(e_{3},e_{3}) = 0$.
Thus the bilinear form $B(\cdot,\cdot)$ in this case is degenerate. Therefore there are no quadratic Novikov algebra structures on Cases  (D3), (D5), (D6) ($l \neq -\frac{1}{2}$) in \cite[Table 5]{BM1}.
	
{\bf Case 6:} Assume that the basis $\{e_1, e_{2},e_{3}\}$ of $(A, \circ)$ satisfies
	  \begin{align*}
	  	e_{1}\circ e_{3} = 0,\;\;e_{3}\circ e_{1} = e_{1},\;\;e_{3}\star e_{2} = me_{2} + ne_{1}\;(m\neq -1),\;\;e_{3}\circ e_{3} = 0.
	  \end{align*}
	  Then we have
	  \begin{align*}
	  	&B(e_{3}\circ e_{1},e_{1}) = B(e_{1},e_{1}) = -B(e_{1},e_{3}\star e_{1}) = -B(e_{1},e_{1}) = 0,\\
	  	&B(e_{3}\circ e_{1},e_{2}) = B(e_{1},e_{2}) = -B(e_{1},e_{3}\star e_{2}) = -mB(e_{1},e_{2}) = 0,\\
	  	&B(e_{3}\circ e_{1},e_{3}) = B(e_{1},e_{3}) = -B(e_{1},e_{3}\star e_{3}) = 0.
	  \end{align*}
Thus the bilinear form $B(\cdot,\cdot)$ in this case is degenerate. Therefore there are no quadratic Novikov algebra structures on Cases (A11)( $l \neq -1$), (A12), (A13) in \cite[Table 2]{BM1}.
	
{\bf Case 7:}	For Cases (A8), (A10), (A11)\;($l= -1$) in \cite[Table 2]{BM1}, it is direct to check that there are no quadratic Novikov algebra structures on these cases.
	
By Cases 1-7, we only need to consider Cases (A7) $(l = -2)$, (C5) $(l = -2)$ and (D6) $(l = -\frac{1}{2})$ given in \cite{BM1}. It is straightforward to check that the symmetric invariant nondegenerate bilinear forms  in these cases are characterized in Table 2. Then the proof is completed.

\delete{Then we get the algebraic structure of $A$ is one of the four cases.
	
	Consider the case $(A7)(l= -2)$. Let $f: A\rightarrow N$ be a linear map such that
	\begin{align*}
		B(e_{1},e_{3}) = \bar{B}(f(e_{1}),f(e_{3})),\\
		B(e_{2},e_{2}) = \bar{B}(f(e_{2}),f(e_{2})),\\
		B(e_{3},e_{3}) = \bar{B}(f(e_{3}),f(e_{3})).
	\end{align*}
	Hence $f$ is an isomorphism of quadratic Novikov algebras. The other three cases are also. This completes the proof.}
\end{proof}

\section{Double extensions of quadratic Novikov algebras}
In this section, we introduce a construction of quadratic Novikov algebras called double extension, and show that a quadratic Novikov algebra containing a nonzero isotropic ideal is isomorphic to a double extension of some lower-dimensional quadratic Novikov algebra.

\begin{lem}\label{th2}
	Let $(A_{1},\circ)$ be a Novikov algebra and $A_{2}$ be a vector space. Let $A_{2}^{\ast}$ be the dual vector space of $A_{2}$ and $\varphi:A_{1} \times A_{1} \rightarrow A_{2}^{\ast}$ be a bilinear map.
	Then $(A_{1}\oplus A_{2}^{\ast},\ast)$ is a Novikov algebra defined by
	\begin{align}
		\label{3.1}(x_{1} + f)\ast (y_{1} + g) = x_{1}\circ y_{1} + \varphi(x_{1},y_{1}),\;\;\; x_1, y_1\in A_1, f, g\in A_2^\ast,
	\end{align}
	 if and only if $\varphi$ satisfies the following conditions:
	\begin{eqnarray}
		&&\label{cent-1}\varphi(x_{1}\circ y_{1},z_{1}) - \varphi(x_{1},y_{1}\circ z_{1}) = \varphi(y_{1}\circ x_{1},z_{1}) - \varphi(y_{1},x_{1}\circ z_{1}),\\
		&&\label{cent-2}\varphi(x_{1}\circ y_{1},z_{1}) = \varphi(x_{1}\circ z_{1},y_{1}),\;
		  x_{1},\; y_{1},\; z_{1} \in A_{1}.
	\end{eqnarray}
\end{lem}
\begin{proof}
	It is straightforward.
\end{proof}
\delete{
 \begin{proof}
	For $x_{1},\; y_{1},\; z_{1} \in A_{1}$ and $f,\; g,\; h \in A_{2}^{\ast}$, we have
	\begin{align*}
		((x_{1} + f)\ast(y_{1} + g))\ast(z_{1} + h) =& (x_{1}\circ y_{1} + \varphi(x_{1},y_{1}))\ast(z_{1} + h) \\
		=& (x_{1}\circ y_{1})\circ z_{1} + \varphi(x_{1}\circ y_{1},z_{1}),\\
		((x_{1} + f)\ast(z_{1} + h))\ast(y_{1} + g) =& (x_{1}\circ z_{1} + \varphi(x_{1},z_{1}))\ast(y_{1} + h) \\
		=& (x_{1}\circ z_{1})\circ y_{1} + \varphi(x_{1}\circ z_{1},y_{1}).
	\end{align*}
	Then since
	\begin{align*}
		(x_{1}\circ y_{1})\circ z_{1} = (x_{1}\circ z_{1})\circ y_{1},
	\end{align*}
	 we have
	\begin{align*}
		((x_{1} + f)\ast(y_{1} + g))\ast(z_{1} + h) = ((x_{1} + f)\ast(z_{1} + h))\ast(y_{1} + g).
	\end{align*} if and only if
	\begin{align*}
		\varphi(x_{1}\circ y_{1},z_{1}) = \varphi(x_{1}\circ z_{1},y_{1}).
	\end{align*}
	
	Next
	\begin{align*}
		((x_{1} + f)\ast(y_{1} + g))\ast(z_{1} + h) &= (x_{1}\circ y_{1} +\varphi(x_{1},y_{1}))\ast(z_{1} + h)
		\\&= (x_{1}\circ y_{1})\circ z_{1} + \varphi(x_{1}\circ y_{1},z_{1}),\\
		-(x_{1} + f)\ast((y_{1} + g)\ast(z_{1} + h)) &= -(x_{1} + f)\ast(y_{1}\circ z_{1} + \varphi(y_{1},z_{1}))
		\\&= - (x_{1}\circ (y_{1}\circ z_{1}) + \varphi(x_{1},y_{1}\circ z_{1})), \\
		((y_{1} + g)\ast(x_{1} + f))\ast(z_{1} + h) &= (y_{1}\circ z_{1} + \varphi(y_{1},z_{1}))\ast(z_{1} + h)
		\\&= (y_{1}\circ x_{1})\circ z_{1} + \varphi(y_{1}\circ x_{1},z_{1}), \\
		-(y_{1} + g)\ast((x_{1} + f)\ast(z_{1} + h)) &= -(y_{1} + g)\ast((x_{1}\circ z_{1} + \varphi(x_{1},z_{1})))
		\\&= -(y_{1}\circ (x_{1}\circ z_{1}) + \varphi(y_{1},x_{1}\circ z_{1})).
	\end{align*}
	Then since
	\begin{align*}
		(x_{1}\circ y_{1})\circ z_{1} - x_{1}\circ (y_{1}\circ z_{1}) &=
		(y_{1}\circ x_{1})\circ z_{1} - y_{1}\circ (x_{1}\circ z_{1}),
	\end{align*}
	we have
	\begin{align*}
		((x + f)\ast(y + g))\ast(z + h) - (x + f)\ast((y + g)\ast(z + h)) =\\ ((y + g)\ast(x + f))\ast(z + h) - (y + g)\ast((x + f)\ast(z + h))
	\end{align*} if and only if
	\begin{align*}
		\varphi(x_{1}\circ y_{1},z_{1}) - \varphi(x_{1},y_{1}\circ z_{1}) &=
		\varphi(y_{1}\circ x_{1},z_{1}) - \varphi(y_{1},x_{1}\circ z_{1}).
	\end{align*}
	Hence $(A_{1}\oplus A_{2}^{\ast},\ast)$ is a Novikov algebra if and only if \begin{align*}
		\varphi(x_{1}\circ y_{1},z_{1}) &= \varphi(x_{1}\circ z_{1},y_{1}),\\
		\varphi(x_{1}\circ y_{1},z_{1}) - \varphi(x_{1},y_{1}\circ z_{1}) &=
		\varphi(y_{1}\circ x_{1},z_{1}) - \varphi(y_{1},x_{1}\circ z_{1}).
	\end{align*}
	
	Moreover, for any $x_{1} + f \in A_{1}\oplus A_{2}^{\ast}$ we take $0 + g \in A_{1}\oplus A_{2}^{\ast}$, then
	\begin{align*}
		(0 + g)\ast(x + f) = 0\circ x + \varphi(0,x) = 0
	\end{align*}
	thus $A_{2}^{\ast}$ is the centre of $A_{1}\oplus A_{2}^{\ast}$.
\end{proof}}

\begin{lem}\label{de}
	Let $(A_{1},\circ)$ and $(A_{2}, \circ)$ be Novikov algebras, $(A_{1} \oplus A_{2}^{\ast},\ast)$ be the Novikov algebra defined by Eq. (\ref{3.1}). Let $\eta$, $\eta^{\prime} : A_{2} \rightarrow \text{End}_{\bf k}(A_{1}\oplus A_{2}^{\ast})$ be two linear maps and
	$\epsilon: A_{2}\times A_{2} \rightarrow A_{1}\oplus A_{2}^{\ast}$ be a bilinear map. Then $(A_{2}\oplus A_{1}\oplus A_{2}^{\ast},\cdot)$ is a Novikov algebra given by
	\begin{eqnarray*}
		(x_{2} + x_{1} +f)\cdot(y_{2} + y_{1} + g) &=& x_{2}\circ y_{2} + \epsilon(x_{2},y_{2}) + (x_{1} + f)\ast(y_{1} + g) + \eta(x_{2})(y_{1} + g) \\
&&\quad+ \eta^{\prime}(y_{2})(x_{1} + f),\;\;x_1, y_1\in A_1, x_2, y_2\in A_2, f, g\in A_2^\ast,
	\end{eqnarray*}
	if and only if $\eta$, $\eta^{\prime}$ and $\epsilon$ satisfy the following conditions:
	\begin{align}
		\label{3.2.1}&\eta(x_{2})((x_{1} + f)\ast(y_{1} + g)) \\ \nonumber =& (\eta(x_{2})(x_{1} + f) - \eta^{\prime}(x_{2})(x_{1}+f))\ast (y_{1} + g) + (x_{1} + f)\ast(\eta(x_{2})(y_{1} + g)),\\
		\label{3.2.2}&\eta^{\prime}(x_{2})((x_{1} + f)\ast(y_{1} + g) - (y_{1} + g)\ast (x_{1} + f)) \\\nonumber =& (x_{1} + f)\ast(\eta^{\prime}(x_{2})(y_{1} + g)) - (y_{1} + g)\ast(\eta^{\prime}(x_{2})(x_{1} + f)),\\
		\label{3.2.3}&\eta^{\prime}(x_{2}\circ y_{2})(x_{1} + f) \\\nonumber =& \eta^{\prime}(y_{2})(\eta^{\prime}(x_{2})(x_{1} + f) - \eta(x_{2})(x_{1} + f)) + \eta(x_{2})\eta^{\prime}(y_{2})(x_{1} + f) - (x_{1} + f)\ast\epsilonup(x_{2},y_{2}),\\
		\label{3.2.4}&\eta(x_{2}\circ y_{2} - y_{2}\circ x_{2})(x_{1} + f) \\\nonumber =& (\eta(x_{2})\eta(y_{2}) - \eta(y_{2})\eta(x_{2}))(x_{1} + f) - (\epsilon(x_{2},y_{2}) - \epsilon(y_{2},x_{2}))\ast(x_{1} + f),\\
		\label{3.2.5}&\epsilon(x_{2}\circ y_{2} , z_{2}) - \epsilon(x_{2},y_{2}\circ z_{2}) - \epsilon(y_{2}\circ x_{2} , z_{2}) + \epsilon(y_{2},x_{2}\circ z_{2}) \\ \nonumber =& -\eta^{\prime}(z_{2})(\epsilon(x_{2},y_{2}) - \epsilon(y_{2},x_{2})) + \eta(x_{2})\epsilon(y_{2},z_{2}) - \eta(y_{2})\epsilon(x_{2},z_{2}),\\
		\label{3.2.6}&(\eta(x_{2})(x_{1} + f))\ast(y_{1} + g) = (\eta(x_{2})(y_{1} + g))\ast(x_{1} + f), \\	
		\label{3.2.7}&(\eta^{\prime}(x_{2})(x_{1} +f))\ast(y_{1} + g) = \eta^{\prime}(x_{2})((x_{1} + f)\ast(y_{1} + g)),\\
		\label{3.2.8}&\eta^{\prime}(y_{2})\eta^{\prime}(x_{2})(x_{1} + f) = \eta^{\prime}(x_{2})\eta^{\prime}(y_{2})(x_{1} + f),\\
		\label{3.2.9}&\eta^{\prime}(y_{2})(\eta(x_{2})(x_{1} + f)) = \epsilon(x_{2},y_{2})\ast(x_{1} + f) + \eta(x_{2}\circ y_{2})(x_{1} + f),\\
		\label{3.2.10}&\eta^{\prime}(z_{2})\epsilon(x_{2},y_{2}) + \epsilon(x_{2}\circ y_{2},z_{2}) =  \eta^{\prime}(y_{2})\epsilon(x_{2},z_{2}) + \epsilon(x_{2}\circ z_{2},y_{2}),
	\end{align}
for all $x_{1}$, $ y_{1} \in A_{1}$ , $f$, $ g \in A_{2}^{\ast}$ and $x_{2}$, $ y_{2}$, $ z_{2} \in A_{2}$.
\end{lem}
\begin{proof}
	It follows from \cite[Example 3.7]{Hong} by letting $(A,\circ)= (A_{1} \oplus A_{2}^{\ast}, \ast)$, $(V, \cdot) = (A_{2}, \circ)$, $l_V=\eta$, $r_V=\eta^{\prime}$ and $f=\epsilon$.
\end{proof}

\delete{\begin{defi}\cite{SCM}
	Let $A_{2}$ be a Novikov algebra and $V$ be a vector space. Let $\omega: A_{2}\times A_{2} \rightarrow V$ be a bilinear map and $\delta,\delta^{\prime}: A_{2} \rightarrow End(V)$ be two linear maps. Then $\omega$ is called {\bf 2-cocycle} if
	\begin{align*}
		&\omega(x_{2},y_{2}\circ z_{2})  + \omega(y_{2}\circ x_{2},z_{2}) +\delta(x_{2})\omega(y_{2},z_{2}) + \delta^{\prime}(z_{2})\omega(y_{2},x_{2})
		\\=& \omega(x_{2}\circ y_{2},z_{2}) + \omega(y_{2},x_{2}\circ z_{2}) + \delta^{\prime}(z_{2})\omega(x_{2},y_{2}) + \delta(y_{2})\omega(x_{2},z_{2}), \\
		&\omega(x_{2}\circ y_{2},z_{2}) + \delta^{\prime}(z_{2})\omega(x_{2},y_{2}) =
		 \omega(x_{2}\circ z_{2},y_{2}) + \delta^{\prime}(y_{2})\omega(x_{2},z_{2}).
	\end{align*}
\end{defi}}

Next, we present a general construction of quadratic Novikov algebras.
\begin{thm}\label{th1}
	Let $(A_{1},\circ,B_{1}(\cdot,\cdot))$ be a quadratic Novikov algebra, $(A_{2},\circ)$ be a Novikov algebra with a symmetric bilinear form $\tau(\cdot,\cdot)$. Let  $\mu$, $\mu^{\prime}: A_{2} \rightarrow \text{End}_{\bf k}(A_{1})$ be linear maps, $\varphi: A_{1}\times A_{1} \rightarrow A_{2}^{\ast}$, $v: A_{2}\times A_{1} \rightarrow A_{2}^{\ast}$, $v^{\prime}: A_{1}\times A_{2} \rightarrow A_{2}^{\ast}$ and $\lambda:A_{2}\times A_{2} \rightarrow A_{1}$, $\gamma:A_{2}\times A_{2} \rightarrow A_{2}^{\ast}$ be bilinear maps.
Then $(A_{2}\oplus A_{1}\oplus A_{2}^{\ast},\cdot,B(\cdot,\cdot))$ is a quadratic Novikov algebra defined by
		\begin{align*}
		(x_{2} + x_{1} + f)\cdot(y_{2} + y_{1} +g) =& x_{2}\circ y_{2} + \lambda(x_{2},y_{2}) + \gamma(x_{2},y_{2}) + x_{1}\circ y_{1} + \varphi(x_{1},y_{1})+ \mu(x_{2})(y_{1}) \\ &+ v(x_{2},y_{1}) + L^{\ast}_{\star}(x_{2})g  + \mu^{\prime}(y_{2})(x_{1}) + v^{\prime}(x_{1},y_{2}) - R^{\ast}_{\circ}(y_{2})f,\\
		B(x_{2} + x_{1} + f,y_{2} + y_{1} + g) =& \tau(x_{2},y_{2}) + B_{1}(x_{1},y_{1}) + f(y_{2}) + g(x_{2}),
	\end{align*}
for all $x_{2}$, $y_{2}\in A_{2}$, $ x_{1}$, $y_{1}\in A_{1}$ and $f$, $g\in A_{2}^{\ast}$ if and only if
	\begin{align}
		\label{3.4.1}&\varphi (x_{1},y_{1})(x_{2}) = -B_{1}(y_{1},\mu(x_{2})x_{1} + \mu^{\prime}(x_{2})x_{1}),\\
		\label{3.4.2}&B_{1}(\mu(x_{2})x_{1},y_{1}) = -B_{1}(x_{1},\mu(x_{2})y_{1} + \mu^{\prime}(x_{2})y_{1}),\\
		\label{3.4.3}&(\mu(x_{2})y_{1})\circ z_{1} = (\mu(x_{2})z_{1})\circ y_{1}, \\
		\label{3.4.4}&\mu^{\prime}(x_{2})x_{1}\circ y_{1} = \mu^{\prime}(x_{2})(x_{1}\circ y_{1}),\\
		\label{3.4.5}&\mu^{\prime}(x_{2})(x_{1}\circ y_{1} - y_{1}\circ x_{1}) = x_{1}\circ \mu^{\prime}(x_{2})y_{1} - y_{1}\circ \mu^{\prime}(x_{2})x_{1},\\
		\label{3.4.6}&\mu(x_{2})(x_{1}\circ y_{1}) = (\mu(x_{2})x_{1} - \mu^{\prime}(x_{2})x_{1})\circ y_{1} + x_{1}\circ \mu(x_{2})y_{1},\\
		\label{3.4.7}&\mu^{\prime}(x_{2})\mu^{\prime}(y_{2}) = \mu^{\prime}(y_{2})\mu^{\prime}(x_{2}),\\	
		\label{3.4.8}&\mu^{\prime}(y_{2})\mu(x_{2})x_{1} = \lambda(x_{2},y_{2})\circ x_{1} + \mu(x_{2}\circ y_{2})x_{1},\\		
		\label{3.4.9}&(\mu(x_{2}\circ y_{2}) - \mu(y_{2}\circ x_{2}))x_{1} \\ \nonumber =& (\mu(x_{2})\mu(y_{2}) - \mu(y_{2})\mu(x_{2}))x_{1} - (\lambda(x_{2},y_{2}) - \lambda(y_{2},x_{2}))\circ x_{1},\\
		\label{3.4.10}&\mu(x_{2})\mu^{\prime}(y_{2}) = \mu^{\prime}(y_{2})\mu(x_{2}),\\
		\label{3.4.11}&x_{1}\circ \lambda(y_{2},x_{2}) + \mu^{\prime}(y_{2}\circ x_{2})x_{1} = \mu^{\prime}(x_{2})\mu^{\prime}(y_{2})x_{1},\\
		\label{3.4.12}&\gamma(x_{2}\circ y_{2},z_{2}) - R^{\ast}_{\circ}(z_{2})\gamma(x_{2},y_{2}) + v^{\prime}(\lambda(x_{2},y_{2}),z_{2}) \\ \nonumber =& \gamma(x_{2}\circ z_{2},y_{2}) - R^{\ast}_{\circ}(y_{2})\gamma(x_{2},z_{2}) + v^{\prime}(\lambda(x_{2},z_{2}),y_{2}),\\
		\label{3.4.13}&\gamma(x_{2}\circ y_{2},z_{2}) - \gamma(x_{2},y_{2}\circ z_{2}) - \gamma(y_{2}\circ x_{2},z_{2})+ \gamma(y_{2},x_{2}\circ z_{2}) \\ \nonumber =& -v^{\prime}(\lambda(x_{2},y_{2}) - \lambda(y_{2},x_{2}),z_{2})) + R^{\ast}_{\circ}(z_{2})(\gamma(x_{2},y_{2}) - \gamma(y_{2},x_{2})) + v(x_{2},\lambda(y_{2},z_{2})) \\ \nonumber &+ L^{\ast}_{\star}(x_{2})\gamma(y_{2},z_{2}) - v(y_{2},\lambda(x_{2},z_{2})) - L^{\ast}_{\star}(y_{2})\gamma(x_{2},z_{2}) ,\\
		\label{3.4.14}& \tau(x_{2}\circ y_{2},z_{2}) + \gamma(x_{2},y_{2})(z_{2}) = - \tau(y_{2},x_{2}\star z_{2}) -\gamma(x_{2},z_{2})(y_{2}) - \gamma(z_{2},x_{2})(y_{2}),\\
		\label{3.4.15}&v(x_{2},y_{1})(z_{2}) = -B_{1}(y_{1},\lambda(x_{2},z_{2}) + \lambda(z_{2},x_{2})),\\
		\label{3.4.16}&v(x_{2},z_{1})(y_{2}) + v^{\prime}(z_{1},x_{2})(y_{2}) = -B_{1}(\lambda(x_{2},y_{2}),z_{1}),\\
		\label{3.4.17}&\lambda(x_{2}\circ y_{2},z_{2}) + \mu^{\prime}(z_{2})\lambda(x_{2},y_{2}) = \lambda(x_{2}\circ z_{2},y_{2}) + \mu^{\prime}(y_{2})\lambda(x_{2},z_{2}),\\
		\label{3.4.18}&\lambda(x_{2}\circ y_{2},z_{2}) - \lambda(x_{2},y_{2}\circ z_{2}) - \lambda(y_{2}\circ x_{2},z_{2}) + \lambda(y_{2},x_{2}\circ z_{2}) \\ \nonumber =& -\mu^{\prime}(z_{2})(\lambda(x_{2},y_{2}) - \lambda(y_{2},x_{2})) + \mu(x_{2})\lambda(y_{2},z_{2}) - \mu(y_{2})\lambda(x_{2},z_{2}),
	\end{align}
	for all $x_{2}$, $y_{2}$, $z_{2}\in A_{2}$ and $ x_{1}$, $y_{1}$, $z_{1}\in A_{1}$.
\end{thm}
\begin{proof}
	By Lemma \ref{de}, $(A_{2}\oplus A_{1}\oplus A_{2}^{\ast}, \cdot)$ is a Novikov algebra if and only if  Eqs. (\ref{cent-1})-(\ref{3.2.10}) hold, where
	\begin{align*}
		\epsilon(x_{2},y_{2}) &= \lambda(x_{2},y_{2}) + \gamma(x_{2},y_{2}),\\
		\eta(x_{2})(y_{1} + g) &= \mu(x_{2})y_{1} + v(x_{2},y_{1}) + L^{\ast}_{\star}(x_{2})g, \\
		\eta^{\prime}(x_{2})(y_{1} + g) &= \mu^{\prime}(x_{2})y_{1} + v^{\prime}(y_{1},x_{2}) - R^{\ast}_{\circ}(x_{2})g,\;x_2, y_2\in A_2,\;x_1,y_1\in A_1,\;f, g\in A_2^\ast.
	\end{align*}
	
Note that
	\begin{align*}
		&\eta(x_{2})((x_{1} + f)\ast(y_{1} + g)) = \eta(x_{2})(x_{1}\circ y_{1} + \varphi(x_{1},y_{1})) \\=& \mu(x_{2})(x_{1}\circ y_{1}) + v(x_{2},x_{1}\circ y_{1}) + L^{\ast}_{\star}(x_{2})\varphi(x_{1},y_{1}),\end{align*}
and
\begin{align*}
		&(\eta(x_{2})(x_{1} + f) - \eta^{\prime}(x_{2})(x_{1} + f))\ast(y_{1} + g) + (x_{1} + f)\ast(\eta(x_{2})(y_{1} + g)) \\=& (\mu(x_{2})x_{1} + v(x_{2},x_{1}) + L^{\ast}_{\star}(x_{2})f - \mu^{\prime}(x_{2})x_{1} - v^{\prime}(x_{1},x_{2}) + R^{\ast}_{\circ}(x_{2})f)\ast(y_{1} + g) \\&+ (x_{1} + f)\ast(\mu(x_{2})y_{1} + v(x_{2},y_{1}) + L^{\ast}_{\star}(x_{2})g) \\=& (\mu(x_{2})x_{1} - \mu^{\prime}(x_{2})x_{1})\circ y_{1} + \varphi(\mu(x_{2})x_{1} - \mu^{\prime}(x_{2})x_{1},y_{1}) + x_{1}\circ \mu(x_{2})y_{1} + \varphi(x_{1},\mu(x_{2})y_{1}).
	\end{align*}
	Therefore, Eq. (\ref{3.2.1}) holds if and only if Eq. (\ref{3.4.6}) and the following equality hold:
	\begin{eqnarray}\label{eq-pf}
		v(x_{2},x_{1}\circ y_{1}) + L^{\ast}_{\star}(x_{2})\varphi(x_{1},y_{1}) = \varphi(\mu(x_{2})x_{1} - \mu^{\prime}(x_{2})x_{1},y_{1}) + \varphi(x_{1},\mu(x_{2})y_{1}).
	\end{eqnarray}

 Similarly, by the fact that $(A_2^\ast, L_\star^\ast, -R_\circ^\ast)$ is a representation of $(A_2, \circ)$, it is direct to check that Eq. (\ref{3.2.2}) holds if and only if Eq. (\ref{3.4.5}) and the following equality hold:
	\begin{align}\label{pf1}
		&\varphi(x_{1},\mu^{\prime}(x_{2})y_{1}) - \varphi(y_{1},\mu^{\prime}(x_{2})x_{1}) \\ \nonumber=&
		v^{\prime}(x_{1}\circ y_{1},x_{2}) - R^{\ast}_{\circ}(x_{2})\varphi(x_{1},y_{1}) - v^{\prime}(y_{1}\circ x_{1},x_{2}) + R^{\ast}_{\circ}(x_{2})\varphi(y_{1},x_{1});
	\end{align}
Eq. (\ref{3.2.3}) holds if and only if the following equalities hold:
	\begin{align}
		\label{pf2}\mu^{\prime}(x_{2}\circ y_{2})x_{1} =& \mu^{\prime}(y_{2})(\mu^{\prime}(x_{2})x_{1} - \mu(x_{2})x_{1}) +  \mu(x_{2})\mu^{\prime}(y_{2})x_{1} - x_{1}\circ \lambda(x_{2},y_{2}),\\
		\label{pf3}v^{\prime}(x_{1},x_{2}\circ y_{2}) =& v^{\prime}(\mu^{\prime}(x_{2})x_{1} - \mu(x_{2})x_{1},y_{2}) - R^{\ast}_{\circ}(y_{2})(v^{\prime}(x_{1},x_{2}) - v(x_{2},x_{1})) \\ \nonumber&+ v(x_{2},\mu^{\prime}(y_{2})x_{1}) + L^{\ast}_{\star}(x_{2})(v^{\prime}(x_{1},y_{2})) - \varphi(x_{1},\lambda(x_{2},y_{2}));
	\end{align}
Eq. (\ref{3.2.4}) holds if and only if Eq. (\ref{3.4.9}) and  the following equality hold:
	\begin{align}\label{pf4}
		v(x_{2}\circ y_{2} - y_{2}\circ x_{2},x_{1}) =& v(x_{2},\mu(y_{2})x_{1}) + L^{\ast}_{\star}(x_{2})(v(y_{2},x_{1})) - v(y_{2},\mu(x_{2})x_{1}) \\ \nonumber& - L^{\ast}_{\star}(y_{2})(v(x_{2},x_{1})) - \varphi(\lambda(x_{2},y_{2}),x_{1}) + \varphi(\lambda(y_{2},x_{2}),x_{1});
	\end{align}	
Eq. (\ref{3.2.5}) holds if and only if Eqs. (\ref{3.4.18}) and (\ref{3.4.13}) hold; Eq. (\ref{3.2.6}) holds if and only if Eq. (\ref{3.4.3}) and the following equality hold:\begin{align}
		\varphi(\mu(x_{2})x_{1},y_{1}) = \varphi(\mu(x_{2})y_{1},x_{1});
	\end{align}
Eq. (\ref{3.2.7}) holds if and only if Eq. (\ref{3.4.4}) and  the following equality hold:
	\begin{align}
		\varphi(\mu^{\prime}(x_{2})x_{1},y_{1}) = v^{\prime}(x_{1}\circ y_{1},x_{2}) - R^{\ast}_{\circ}(x_{2})\varphi(x_{1},y_{1});
	\end{align}	
Eq. (\ref{3.2.8}) holds if and only if Eq. (\ref{3.4.7}) and the following equality hold:
	\begin{align}
		v^{\prime}(\mu^{\prime}(y_{2})x_{1},x_{2}) - R^{\ast}_{\circ}(x_{2})v^{\prime}(x_{1},y_{2}) = v^{\prime}(\mu^{\prime}(x_{2})x_{1},y_{2}) - R^{\ast}_{\circ}(y_{2})v^{\prime}(x_{1},x_{2});
\end{align}
Eq. (\ref{3.2.9}) holds if and only if Eq. (\ref{3.4.8}) and the following equality hold:
	\begin{align}
		\label{pf5}v^{\prime}(\mu(x_{2})x_{1},y_{2}) - R^{\ast}_{\circ}(y_{2})v(x_{2},x_{1}) = \varphi(\lambda(x_{2},y_{2}),x_{1}) + v(x_{2}\circ y_{2},x_{1});
	\end{align}
Eq. (\ref{3.2.10}) holds if and only if Eqs. (\ref{3.4.17}) and (\ref{3.4.12}) holds.

	By the definition of $B(\cdot,\cdot)$, we have
	\begin{align*}
		&B((x_{2} + x_{1} + f)\cdot(y_{2} + y_{1} + g),z_{2} + z_{1} + h) \\=& B(x_{2}\circ y_{2} + \lambda(x_{2},y_{2}) + \gamma(x_{2},y_{2}) + x_{1}\circ y_{1} + \varphi(x_{1},y_{1}) + \mu(x_{2})y_{1} + v(x_{2},y_{1}) + L^{\ast}_{\star}(x_{2})g \\&+ \mu^{\prime}(y_{2})x_{1} + v^{\prime}(x_{1},y_{2}) - R^{\ast}_{\circ}(y_{2})f, z_{2} + z_{1} + h) \\= & \tau(x_{2}\circ y_{2},z_{2}) + B_{1}(\lambda(x_{2},y_{2}) + x_{1}\circ y_{1} + \mu(x_{2})y_{1} + \mu^{\prime}(y_{2})x_{1},z_{1}) + (\gamma(x_{2},y_{2}) + \varphi(x_{1},y_{1}) + v(x_{2},y_{1}) \\&+ L^{\ast}_{\star}(x_{2})g + v^{\prime}(x_{1},y_{2}) - R^{\ast}_{\circ}(y_{2})f)(z_{2}) + h(x_{2}\circ y_{2}),
\end{align*}
and
\begin{align*}
		&-B(y_{2} + y_{1} + g,(x_{2} + x_{1} + f)\star (z_{2} + z_{1} + h)) \\=& - B(y_{2} + y_{1} + g , x_{2}\circ z_{2} + \lambda(x_{2},z_{2}) + \gamma(x_{2},z_{2}) + x_{1}\circ z_{1} + \varphi(x_{1},z_{1}) + \mu(x_{2})z_{1} + v(x_{2},z_{1}) + L^{\ast}_{\star}(x_{2})h \\&+ \mu^{\prime}(z_{2})x_{1} + v^{\prime}(x_{1},z_{2}) - R^{\ast}_{\circ}(z_{2})f+z_{2}\circ x_{2} + \lambda(z_{2},x_{2}) + \gamma(z_{2},x_{2}) + z_{1}\circ x_{1} + \varphi(z_{1},x_{1}) \\
&+ \mu(z_{2})x_{1} + v(z_{2},x_{1}) + L^{\ast}_{\star}(z_{2})f + \mu^{\prime}(x_{2})z_{1} + v^{\prime}(z_{1},x_{2}) - R^{\ast}_{\circ}(x_{2})h) \\=& - \tau(y_{2},x_{2}\circ z_{2}) - B_{1}(y_{1} , \lambda(x_{2},z_{2}) + x_{1}\circ z_{1} + \mu(x_{2})z_{1} + \mu^{\prime}(z_{2})x_{1}) - g(x_{2}\circ z_{2}) - (\gamma(x_{2},z_{2}) + \varphi(x_{1},z_{1}) \\&+ v(x_{2},z_{1}) + L^{\ast}_{\star}(x_{2})h +v ^{\prime}(x_{1},z_{2}) - R^{\ast}_{\circ}(z_{2})f)(y_{2})- \tau(y_{2},z_{2}\circ x_{2}) \\
&-B_{1}(y_{1} , \lambda(z_{2},x_{2}) + z_{1}\circ x_{1} + \mu(z_{2})x_{1} + \mu^{\prime}(x_{2})z_{1}) - g(z_{2}\circ x_{2}) \\
&- (\gamma(z_{2},x_{2}) + \varphi(z_{1},x_{1})+ v(z_{2},x_{1}) + L^{\ast}_{\star}(z_{2})f + v^{\prime}(z_{1},x_{2}) - R^{\ast}_{\circ}(x_{2})h)(y_{2}).
	\end{align*}
Therefore we obtain that $B(\cdot,\cdot)$ satisfies the invariant property if and only if Eqs. (\ref{3.4.1})-(\ref{3.4.2}), Eqs. (\ref{3.4.14})-(\ref{3.4.16}) and the following equalities hold for all $x_1$, $y_1$, $z_1\in A_1$ and $x_2$, $y_2$, $z_2\in A_2$:
	\begin{align}
		\label{th.2.1} B_{1}(\mu^{\prime}(y_{2})x_{1},z_{1}) &= -\varphi(x_{1},z_{1})y_{2} - \varphi(z_{1},x_{1})y_{2},\\
		\label{th.2.2}v^{\prime}(x_{1},y_{2})z_{2} &= -v^{\prime}(x_{1},z_{2})(y_{2}) - v^{\prime}(z_{2},x_{1})(y_{2}).
	\end{align}
By Eqs. (\ref{3.4.1})-(\ref{3.4.2}) and Eqs. (\ref{3.4.15})-(\ref{3.4.16}), it is easy to check that Eqs. (\ref{th.2.1}) and (\ref{th.2.2}) hold.
Therefore, $B(\cdot,\cdot)$ satisfies the invariant property if and only if Eqs. (\ref{3.4.1})-(\ref{3.4.2}) and  (\ref{3.4.14})-(\ref{3.4.16}) hold.

By Eqs. (\ref{3.4.1})-(\ref{3.4.2}) and (\ref{3.4.15})-(\ref{3.4.16}), we have
\begin{align}
	\label{eq-p3}\varphi(x_{1},y_{1})(z_{2}) =& -B_{1}(y_{1},\mu^{\prime}(z_{2})(x_{1}) + \mu(z_{2})(x_{1})) = B_{1}(\mu(z_{2})y_{1},x_{1}),\\
	\label{eq-p1}v^{\prime}(x_{1},x_{2})(y_{2}) =& -v(x_{2},x_{1})(y_{2}) - B_{1}(\lambda(x_{2},y_{2}),x_{1}) \\=& B_{1}(x_{1},\lambda(x_{2},y_{2}) + \lambda(y_{2},x_{2})) - B_{1}(\lambda(x_{2},y_{2}),x_{1}) \nonumber\\=& B_{1}(x_{1},\lambda(y_{2},x_{2})),\nonumber\\
	\label{eq-p2}B_{1}(\mu^{\prime}(x_{2})y_{1},z_{1}) =& - B_{1}(\mu(x_{2})y_{1},z_{1}) - B_{1}(y_{1},\mu(x_{2})z_{1})\\=&  B_{1}(y_{1},\mu(x_{2})z_{1} + \mu^{\prime}(x_{2})z_{1}) - B_{1}(y_{1},\mu(x_{2})z_{1}) \nonumber\\=& B_{1}(y_{1},\mu^{\prime}(x_{2})z_{1}).\nonumber
\end{align}	
By Eqs. (\ref{3.4.3}), (\ref{3.4.4}), (\ref{3.4.6}), (\ref{eq-p3}) and (\ref{eq-p2}), for all $a\in A_2$, we obtain that
\begin{align*}
	\varphi(x_{1}\circ y_{1},z_{1})(a) &= B_{1}(\mu(a)z_{1},x_{1}\circ y_{1})  = B_{1}((\mu(a)z_{1})\circ y_{1},x_{1}) \\ &= B_{1}((\mu(a)y_{1})\circ z_{1},x_{1}) = B_{1}(x_{1}\circ z_{1},\mu(a)y_{1}) \\&= \varphi(x_{1}\circ z_{1},y_{1})(a),
\end{align*}
and
\begin{align*}
	&\varphi(x_{1}\circ y_{1},z_{1})(a) - \varphi(x_{1},y_{1}\circ z_{1})(a) \\=& B_{1}(\mu(a)z_{1},x_{1}\circ y_{1}) - B_{1}(\mu(a)(y_{1}\circ z_{1}),x_{1}) \\=& B_{1}((\mu(a)z_{1})\circ y_{1},x_{1}) - B_{1}((\mu(a)y_{1})\circ z_{1} - (\mu^{\prime}(a)y_{1})\circ z_{1} + y_{1}\circ \mu(a)z_{1},x_{1}) \\=& B_{1}((\mu^{\prime}(a)y_{1})\circ z_{1} - y_{1}\circ (\mu(a)z_{1}),x_{1}) = B_{1}(\mu^{\prime}(a)(y_{1}\circ z_{1}),x_{1}) - B_{1}(x_{1}\circ (\mu(a)z_{1}),y_{1})\\=& B_{1}(y_{1}\circ z_{1},\mu^{\prime}(a)x_{1}) - B_{1}(x_{1}\circ (\mu(a)z_{1}),y_{1}) = B_{1}((\mu^{\prime}(a)x_{1})\circ z_{1},y_{1}) - B_{1}(x_{1}\circ (\mu(a)z_{1}),y_{1}) \\=& B_{1}(y_{1},(\mu(a)z_{1})\circ x_{1}) - B_{1}((\mu(a)x_{1})\circ z_{1} - (\mu^{\prime}(a)x_{1})\circ z_{1} + x_{1}\circ (\mu(a)z_{1}),y_{1}) \\=& \varphi(y_{1}\circ x_{1},z_{1})(a) - \varphi(y_{1},x_{1}\circ z_{1})(a).
\end{align*}
Therefore Eqs. (\ref{cent-1}) and (\ref{cent-2}) hold.

\delete{Therefore, by the discussion above, for finishing the proof, we only need to check that if Eqs. (\ref{3.4.1})-(\ref{3.4.18}) hold, then Eqs. (\ref{eq-pf})-(\ref{pf5}) naturally hold.}

By Eqs. (\ref{3.4.1})-(\ref{3.4.2}),  (\ref{3.4.8}), (\ref{3.4.15})-(\ref{3.4.16}) and (\ref{eq-p1}), for all $a\in A_2$, we obtain
\begin{align*}
	&(\varphi(x_{1},\mu^{\prime}(x_{2})y_{1}) - \varphi(y_{1},\mu^{\prime}(x_{2})x_{1}))(a) \\ =& -B_{1}(\mu^{\prime}(x_{2})y_{1},\mu(a)x_{1} + \mu^{\prime}(a)x_{1}) + B_{1}(\mu^{\prime}(x_{2})x_{1},\mu(a)y_{1} + \mu^{\prime}(a)y_{1}) \\ =& -B_{1}(y_{1},\mu^{\prime}(x_{2})(\mu(a)x_{1} + \mu^{\prime}(a)x_{1})) - B_{1}(y_{1},\mu(a)\mu^{\prime}(x_{2})x_{1})\\ =& -B_{1}(y_{1},\mu^{\prime}(x_{2})\mu(a)x_{1} + \mu^{\prime}(x_{2})\mu^{\prime}(a)x_{1} + \mu^{\prime}(x_{2})\mu(a)x_{1}),
\end{align*}
and
\begin{align*}
	&(v^{\prime}(x_{1}\circ y_{1},x_{2}) - R^{\ast}_{\circ}(x_{2})\varphi(x_{1},y_{1}) - v^{\prime}(y_{1}\circ x_{1},x_{2}) + R^{\ast}_{\circ}(x_{2})\varphi(y_{1},x_{1}))(a)\\
	= & B_{1}(x_{1}\circ y_{1},\lambda(a,x_{2})) - B_{1}(y_{1},\mu(a\circ x_{2})x_{1} + \mu^{\prime}(a\circ x_{2})x_{1}) - B_{1}(y_{1}\circ x_{1},\lambda(a,x_{2})) \\&- B(x_{1},\mu(a\circ x_{2})y_{1} + \mu^{\prime}(a\circ x_{2})y_{1})\\
	=& -B_{1}(y_{1},x_{1}\star\lambda(a,x_{2}) + \mu(a\circ x_{2})x_{1} + \mu^{\prime}(a\circ x_{2})x_{1} + \lambda(a,x_{2})\circ x_{1} + \mu(a\circ x_{2})x_{1}) \\ =& -B_{1}(y_{1},x_{1}\circ\lambda(a,x_{2}) + \mu^{\prime}(a\circ x_{2})x_{1} + 2\mu^{\prime}(x_{2})\mu(a)x_{1} ).
\end{align*}
Then Eq. (\ref{pf1}) holds if and only if Eq. (\ref{3.4.11}) holds. Similarly by Eq. (\ref{3.4.11}), one can check that Eq. (\ref{pf2}) holds if and only if Eq. (\ref{3.4.10}) holds.

By Eqs. (\ref{3.4.8}) and (\ref{3.4.9}), we have
\begin{eqnarray}
\label{eqq}(\mu(x_2)\mu(y_2)-\mu(y_2)\mu(x_2))x_1&=&(\mu(x_2\circ y_2)-\mu(y_2\circ x_2))x_1+(\lambda(x_2,y_2)-\lambda(y_2,x_2))\circ x_1\\
&=&\mu^{\prime}(y_2)\mu(x_2)x_1-\mu^{\prime}(x_2)\mu(y_2)x_1,\;\;x_1\in A_1,\;x_2,y_2\in A_2.\nonumber
\end{eqnarray}

By Eqs. (\ref{def-eq1}), (\ref{3.4.1}), (\ref{3.4.2}), (\ref{3.4.8}), (\ref{3.4.10}),  (\ref{3.4.11}), (\ref{3.4.15}) and (\ref{eqq}), for all $a\in A_2$, we have
	\begin{align*}
		&(v(x_{2},x_{1}\circ y_{1}) + L^{\ast}_{\star}(x_{2})\varphi(x_{1},y_{1}))(a) \\=& -B_{1}(x_{1}\circ y_{1},\lambda(a,x_{2}) + \lambda(x_{2},a)) - B_{1}(y_{1},\mu(-a\star x_{2})x_{1} + \mu^{\prime}(-a\star x_{2})x_{1}) \\=& B_{1}(y_{1},x_{1}\star(\lambda(a,x_{2}) + \lambda(x_{2},a))) + B_{1}(y_{1},\mu(a\star x_{2})x_{1} + \mu^{\prime}(a\star x_{2})x_{1}) \\=& B_{1}(y_{1},\mu^{\prime}(a)\mu^{\prime}(x_{2})x_{1} + \mu^{\prime}(x_{2})\mu^{\prime}(a)x_{1} + \mu^{\prime}(x_{2})\mu(a)x_{1} +\mu^{\prime}(a)\mu(x_{2})x_{1}) \\=& B_{1}(y_{1},\mu^{\prime}(a)\mu^{\prime}(x_{2})x_{1} + \mu^{\prime}(x_{2})\mu^{\prime}(a)x_{1} + \mu(x_{2})\mu^{\prime}(a)x_{1} +\mu^{\prime}(x_{2})\mu(a)x_{1} \\&- \mu(a)\mu(x_{2})x_{1} - \mu^{\prime}(a)\mu(x_{2})x_{1} + \mu(x_{2})\mu(a)x_{1} + \mu^{\prime}(x_{2})\mu(a)x_{1}) \\=& (\varphi(\mu(x_{2})x_{1} - \mu^{\prime}(x_{2})x_{1},y_{1}) + \varphi(x_{1},\mu(x_{2})y_{1}))(a).
	\end{align*}
Therefore, Eq. (\ref{eq-pf}) naturally holds.

By Eqs. (\ref{3.4.1}), (\ref{3.4.2}), (\ref{3.4.15}), (\ref{3.4.17}), (\ref{3.4.18}), (\ref{eq-p1}) and (\ref{eq-p2}), for all $a\in A_2$, we obtain
	\begin{align*}
		&(v^{\prime}(\mu^{\prime}(x_{2})x_{1} - \mu(x_{2})x_{1},y_{2}) - R^{\ast}_{\circ}(y_{2})(v^{\prime}(x_{1},x_{2}) - v(x_{2},x_{1})) + v(x_{2},\mu^{\prime}(y_{2})x_{1}) + L^{\ast}_{\star}(x_{2})v^{\prime}(x_{1},y_{2}) \\&- \varphi(x_{1},\lambda(x_{2},y_{2})))(a) \\=& B_{1}(\mu^{\prime}(x_{2})x_{1} - \mu(x_{2})x_{1},\lambda(a,y_{2})) + B_{1}(x_{1},\lambda(a\circ y_{2},x_{2}) + \lambda(a\circ  y_{2},x_{2}) + \lambda(x_{2},a\circ y_{2})) \\&- B_{1}(\mu^{\prime}(y_{2})x_{1},\lambda(a,x_{2}) + \lambda(x_{2},a)) + B_{1}(x_{1},\lambda(-a\star x_{2},y_{2})) + B_{1}(\lambda(x_{2},y_{2}),\mu(a)x_{1} + \mu^{\prime}(a)x_{1}) \\=& B_{1}(x_{1},\mu^{\prime}(x_{2})\lambda(a,x_{2}) + \mu(x_{2})\lambda(a,x_{2}) + \mu^{\prime}(x_{2})\lambda(a,y_{2}) + \lambda(a\circ y_{2},x_{2}) + \lambda(a\circ y_{2},x_{2}) + \lambda(x_{2},a\circ y_{2})) \\& -B_{1}(x_{1},\mu^{\prime}(y_{2})\lambda(a,x_{2}) + \mu^{\prime}(y_{2})\lambda(x_{2},a) + \lambda(a\star x_{2},y_{2})) + \mu(a)\lambda(x_{2},y_{2}))\\ =& B_{1}(x_{1},\lambda(a,x_{2}\circ y_{2})) = v^{\prime}(x_{1},x_{2}\circ y_{2})(a).
	\end{align*}
	Hence (\ref{pf3}) holds. Similarly, it is direct to check that
	Eqs. (\ref{pf4}) - (\ref{pf5}) hold.
	
	Since $B_{1}(\cdot,\cdot)$ and $\tau(\cdot,\cdot)$ is symmetric, $B(\cdot,\cdot)$ is symmetric. Assume that $x_{2} + x_{1} + f \in (A_{2}\oplus A_{1}\oplus A_{2}^{\ast})^{\perp}$. Then we have $B(x_{2} + x_{1} + f,y_{1}) = B_{1}(x_{1},y_{1}) = 0$ for all $y_{1} \in A_{1}$. Thus $x_{1} = 0$. Then $B(x_{2} + f,g) = g(x_{2}) = 0$ for all $g \in A_{2}^{\ast}$. Hence we get $x_{2} = 0$. Since $B(f,y_{2}) = f(y_{2}) = 0$ for all $y_{2} \in A_{2}$, we obtain $f = 0$. Thus $(A_{2}\oplus A_{1}\oplus A_{2}^{\ast})^{\perp} = {0}$, i.e. $B(\cdot,\cdot)$ is nondegenerate.

Therefore, by the discussion above, we obtain that $(A_{2}\oplus A_{1}\oplus A_{2}^{\ast},\cdot,B(\cdot,\cdot))$ is a quadratic Novikov algebra if and only if Eqs. (\ref{3.4.1})-(\ref{3.4.18}) hold. Then the proof is completed.
\end{proof}
\begin{rmk}
	$(A_{2}\oplus A_{1}\oplus A_{2}^{\ast},\cdot,B(\cdot,\cdot))$ in Theorem \ref{th1} is called a {\bf double extension} of $(A_{1},\circ, B_1(\cdot,\cdot))$ by a Novikov algebra $(A_{2},\circ)$ with a symmetric bilinear form $\tau(\cdot,\cdot)$. \delete{If $\tau=0$, for convenience, we call that this quadratic Novikov algebra is a {\bf double extension} of $(A_{1},\circ, B_1(\cdot,\cdot))$ by a Novikov algebra $(A_{2},\circ)$ and } We denote it by $(A_{2}\oplus A_{1}\oplus A_{2}^{\ast}, B_{1}(\cdot,\cdot),\tau(\cdot,\cdot),\varphi,\mu,\mu^{\prime},v,v^{\prime},\lambda,\gamma)$.
\end{rmk}

\begin{cor}
	Let $(A_{2},\circ)$ be a Novikov algebra with a symmetric bilinear form $\tau(\cdot,\cdot)$ and $\gamma:A_{2}\times A_{2} \rightarrow A_{2}^{\ast}$ be a bilinear map. Then $(A_{2} \oplus A_{2}^{\ast},\diamond,\tilde{B}(\cdot,\cdot))$ is a quadratic Novikov algebra given by
	\begin{align*}
		(x_{2} + f)\diamond(y_{2} + g) &= x_{2}\circ y_{2} + \gamma(x_{2},y_{2}) + L^{\ast}_{\star}(x_{2})g - R^{\ast}_{\circ}(y_{2})f,\\
		\tilde{B}(x_{2} + f,y_{2} + g) &= \tau(x_{2},y_{2}) + f(y_{2}) \hspace{2pt}+\hspace{2pt} g(x_{2}),\;\; x_2, y_2\in A_2,\;\; f, g\in A_2^\ast,
	\end{align*}
	if and only if $\gamma$ satisfies
	\begin{eqnarray}		
&&\gamma(x_{2}\circ y_{2},z_{2}) - R^{\ast}_{\circ}(z_{2})\gamma(x_{2},y_{2}) = \gamma(x_{2}\circ z_{2},y_{2}) - R^{\ast}_{\circ}(y_{2})\gamma(x_{2},z_{2}),\\
&&\gamma(x_{2}\circ y_{2},z_{2}) - \gamma(x_{2},y_{2}\circ z_{2}) - \gamma(y_{2}\circ x_{2},z_{2})+ \gamma(y_{2},x_{2}\circ z_{2}) \\&&\quad= R^{\ast}_{\circ}(z_{2})(\gamma(x_{2},y_{2}) - \gamma(y_{2},x_{2})) +  L^{\ast}_{\star}(x_{2})\gamma(y_{2},z_{2}) - L^{\ast}_{\star}(y_{2})\gamma(x_{2},z_{2}), \nonumber\\
&&\tau(x_{2}\circ y_{2},z_{2}) + \gamma(x_{2},y_{2})(z_{2}) = - \tau(y_{2},x_{2}\star z_{2}) -\gamma(x_{2},z_{2})(y_{2}) - \gamma(z_{2},x_{2})(y_{2}),\;\; x_{2},y_{2},z_{2} \in A_{2}.
	\end{eqnarray}
\end{cor}
\begin{proof}
	It follows from  Theorem \ref{th1} directly by letting $A_1=0$.
\end{proof}
\delete{\begin{rmk}
	$(A_{2} \oplus A_{2}^{\ast},\diamond,\tilde{B}(\cdot,\cdot))$ is called the $T^{\ast}$ extension of $(A_{2},\circ)$.
\end{rmk}}
 \begin{cor}\label{cor-com-dim1}
	 	Let $(A_{2}={\bf k}e,\circ)$ be the Novikov algebra of dimension $1$ defined by $e\circ e= ke$ for some $k\in {\bf k}$ and $A_{2}^{\ast}={\bf k}e^\ast$ where $e^\ast(e)=1$. Let $(A_{1},\circ,B_{1}(\cdot,\cdot))$ be a quadratic Novikov algebra, $h(\cdot,\cdot):A_{1}\times A_{1}\rightarrow {\bf k}$ be a bilinear map, $f,g: A_{1}\rightarrow {\bf k}$, $Q_1$, $Q_2:A_1\rightarrow A_1$ be linear maps, $\alpha\in A_1$, and $t$, $s\in {\bf k}$. Then
$(A_{2}\oplus A_{1}\oplus A_{2}^{\ast}, \cdot, B(\cdot,\cdot))$ is a quadratic Novikov algebra where the binary operation $\cdot$ and the bilinear form $B(\cdot,\cdot)$ are defined by
	 	\begin{align*}
	 		(k_1e + x_{1} + l_1e^{\ast})\cdot(k_2e + y_{1} + l_2e^{\ast}) &= k_1k_2ke + k_1k_2\alpha + x_{1}\circ y_{1}+k_1Q_{1}(y_{1})+k_2 Q_2(x_1)\\
&=+(k_1k_2s+h(x_1,y_1)+k_1f(y_1)-2k_1l_2k+g(x_1)k_2+k_2l_1k)e^\ast,\\
	 		B_1(k_1e + x_{1} + l_1e^{\ast},k_2e + y_{1} + l_2e^{\ast}) &= k_1k_2t + B_{1}(x_{1},y_{1}) + l_1k_2+k_1l_2,\;\;k_1, k_2, l_1, l_2\in {\bf k},
	 	\end{align*} if and only if
	 	\begin{align*}
            &kt + s= 0,\\
	 		&h(x_{1},y_{1})= -B_{1}(y_{1},Q_{1}(x_{1})+ Q_{2}(x_{1})), \\
	 		&B_1(Q_{1}(x_{1}),y_{1}) = -B_{1}(x_{1},Q_{1}(y_{1}) + Q_{2}(y_{1})), \\
	 		&Q_{1}(x_{1})\circ y_{1} = Q_{1}(y_{1})\circ x_{1},\\
	 		&Q_{2}(x_{1})\circ y_{1} = Q_{2}(x_{1}\circ y_{1}),\\
	 		&Q_{2}(x_{1}\circ y_{1} - y_{1}\circ x_{1}) = x_{1}\circ Q_{2}(y_{1}) - y_{1}\circ Q_{2}(x_{1}),\\
	 		&Q_{1}(x_{1}\circ y_{1}) = (Q_{1}(x_{1}) - Q_{2}(x_{1}))\circ y_{1} + x_{1}\circ Q_{1}(y_{1}),\\
	 		&Q_{2}(Q_{1}(x_{1})) = \alpha\circ x_{1} + kQ_{1}(x_{1}),\\
	 		&Q_{1}Q_{2}= Q_{2}Q_{1},\\
	 		&Q_{2}^{2}(x_{1}) = x_{1}\circ \alpha + kQ_{2}(x_{1}),\\
   &f(x_{1}) = -2B_1(x_{1},\alpha),\\
	 		&f(x_{1}) + g(x_{1}) = -B_1(\alpha,x_{1}), \; x_{1},y_{1} \in A_{1}.
	 	\end{align*}
	 	
	 	\delete{ In particular, if $k = 0$, we get $L^{\ast}_{\star}(e)e^{\ast}(e) = -e^{\ast}(e\star e) = 0$. Then $L^{\ast}_{\star}(e)e^{\ast} = 0$. Similarily, $- R^{\ast}_{\circ}(e)e^{\ast} = 0$. Thus $A_{2}\oplus A_{1}\oplus A_{2}^{\ast}$ is the double extension of $(A_{1},\circ)$ by $A_{2}$ which is defined by
	 	 \begin{align*}
	 	 	(e + x_{1} + e^{\ast})\cdot(e + y_{1} + e^{\ast}) &= \alpha + se^{\ast} + x_{1}\circ y_{1} + \varphi(x_{1},y_{1}) + Q_{1}y_{1} + f(y_{1})e^{\ast}   + Q_{2}x_{1} + g(x_{1})e^{\ast},\\
	 	 	B(e + x_{1} + e^{\ast},e + y_{1} + e^{\ast}) &= \beta + B_{1}(x_{1},y_{1}) + 2e^{\ast}(e),
	 	 \end{align*} if and only if
	 	 \begin{align*}
	 	 	&\varphi(x_{1},y_{1})(e) = -B_{1}(y_{1},Q_{1}x_{1} + Q_{2}x_{1}), \\
	 	 	&B_1(Q_{1}x_{1},y_{1}) = -B_{1}(x_{1},Q_{1}y_{1} + Q_{2}X_{1}), \\
	 	 	&(Q_{1}x_{1})\circ y_{1} = (Q_{1}y_{1})\circ x_{1},\\
	 	 	&(Q_{2}x_{1})\circ y_{1} = Q_{2}(x_{1}\circ y_{1}),\\
	 	 	&Q_{2}(x_{1}\circ y_{1} - y_{1}\circ x_{1}) = x_{1}\circ(Q_{2}y_{1}) - y_{1}\circ(Q_{2}x_{1}),\\
	 	 	&Q_{1}(x_{1}\circ y_{1}) = (Q_{1}x_{1} - Q_{2}x_{1})\circ y_{1} + x_{1}\circ (Q_{1}y_{1}),\\
	 	 	&Q_{2}Q_{1}x_{1} = \alpha\circ x_{1} \\
	 	 	&Q_{1}Q_{2}= Q_{2}Q_{1},\\
	 	 	&(Q_{2})^{2}x_{1} = x_{1}\circ \alpha\\
	 	 	&s = l = 0,\\
	 	 	&f(x_{1})e^{\ast}(e) = -B_1(x_{1},2\alpha),\\
	 	 	&f(x_{1})e^{\ast}(e) + g(x_{1})e^{\ast}(e) = -B_1(\alpha,x_{1}).
	 	 \end{align*}}
	 \end{cor}
	\begin{proof}
Since the dimension of $A_2$ is $1$, we can set $\lambda(e,e)= \alpha$, $\gamma(e,e) = se^{\ast}$, $\tau(e,e) = t$, $\varphi(x_{1},y_{1}) = h(x_{1},y_{1})e^{\ast}$, $ \mu(e)=Q_1$, $\mu^{\prime}(e)=Q_2$, $ v(e,x_{1}) = f(x_{1})e^{\ast}$, $v^{\ast}(x_{1},e) = g(x_{1})e^{\ast}$ in Theorem \ref{th1}, where $h(\cdot,\cdot):A_{1}\times A_{1}\rightarrow {\bf k}$ is a bilinear map, $f,g: A_{1}\rightarrow {\bf k}$, $Q_1$, $Q_2:A_1\rightarrow A_1$ are linear maps, $\alpha\in A_1$, and $t$, $s\in {\bf k}$. Then this conclusion follows directly from Theorem \ref{th1}.	
\end{proof}
\begin{rmk}\label{rmk-dou-1}
By Corollary \ref{cor-com-dim1}, for determining the double extensions of $(A_1, \circ, B_1(\cdot,\cdot))$ by $(A_2={\bf k}e, \circ)$ with a symmetric bilinear form $\tau(\cdot,\cdot)$, we only need to compute
 $h(\cdot,\cdot)$, $f$, $g$, $Q_1$, $Q_2$, $\alpha$, $t$, $s$ in Corollary \ref{cor-com-dim1}.

\end{rmk}

Next, we present an example to construct quadratic Novikov algebras by double extensions of quadratic Novikov algebras.
	 \begin{ex}
	 	Let $(A_{1},\circ)$ be a Novikov algebra with a basis $\{e_{1},e_{2}\}$ defined by
	 	\begin{align*}
	 		e_{1}\circ e_{1} = 0, e_{1}\circ e_{2} = e_{1},
	 		e_{2}\circ e_{1} = -2e_{1}, e_{2}\circ e_{2} = e_{2}.
	 	\end{align*}
	 Define a bilinear map $B_1(\cdot,\cdot)$ on $(A_{1},\circ)$ by
	 	\begin{align*}
	 		B_1(e_{1},e_{2}) = B_1(e_{2},e_{1}) = 1,\; B_1(e_{1},e_{1}) = B_1(e_{2},e_{2}) = 0.
	 	\end{align*}
Then $(A_1, \circ, B_1(\cdot,\cdot))$ is a quadratic Novikov algebra by Theorem \ref{2dim}.

Let $(A_{2}={\bf k}e,\circ)$ be the $1$-dimensional Novikov algebra defined by $e\circ e = e$ and $A_{2}^{\ast}={\bf k}e^\ast$ where $e^{\ast}(e) = 1$. By Remark \ref{rmk-dou-1}, we need to compute
 $h(\cdot,\cdot)$, $f$, $g$, $Q_1$, $Q_2$, $\alpha$, $t$, $s$ in Corollary \ref{cor-com-dim1}. Note that $k=1$ and $t+s=0$. Then by some direct computations, we  obtain that
 \begin{eqnarray*}
 h(e_1,e_1)=h(e_2,e_2)=0,\;\;h(e_1,e_2)=-1,\;\;h(e_2,e_1)=2,\;\alpha=2e_2,\\
 Q_1(e_1)=2e_1,\;Q_1(e_2)=-e_2,\;\;Q_2(e_1)=-e_1,\;Q_2(e_2)=-e_2,\\
 f(e_1)=-4,\;\;f(e_2)=0,\;\;g(e_1)=2,\;\;g(e_2)=0.
\end{eqnarray*}
Then by Corollary \ref{cor-com-dim1}, $({\bf k}e \oplus A_{1}\oplus {\bf k}e^\ast,\cdot, B(\cdot,\cdot))$ is a quadratic Novikov algebra defined by the nonzero products and the symmetric bilinear form $B(\cdot,\cdot)$ as follows:
\begin{eqnarray*}
&&e\cdot e=e+2e_2+se^{\ast},\;e\cdot e^\ast=-2e^\ast,\;e^\ast \cdot e=e^\ast,\;e\cdot e_1=2e_1-4e^\ast,\\
&&e\cdot e_2=-e_2,\; e_1\cdot e=-e_1+2e^\ast,\; e_2\cdot e=-e_2,\\
&&e_1\cdot e_2=e_1-e^\ast,\;e_2\cdot e_1=-2e_1+2e^\ast,\;e_2\cdot e_2=e_2,\\
&&B(e, e^\ast)=B(e_1,e_2)=1,B(e,e)= -s,\\
&&B(e_1,e_1)=B(e_2,e_2)=B(e,e_1)=B(e,e_2)=B(e_1,e^\ast)=B(e_2,e^\ast)=0,
\end{eqnarray*}
where $s\in {\bf k}$.
\delete{$\mu,\mu^{\prime}:A_{2}\rightarrow End_{\bf k}(A_{1})$ be the linear maps such that
	 	\begin{align*}
	 		&\mu(e)(e_{1}) = 2e_{1},\; \mu(e)(e_{2}) = -e_{2},\\
	 		&\mu^{\prime}(e)(e_{1}) = -e_{1},\; \mu^{\prime}(e)(e_{2}) = -e_{2}.
	 	\end{align*}
	 	
	 	Let $\varphi:A_{1}\times A_{1} \rightarrow A_{2}^{\ast}, \lambda:A_{2}\times A_{2}\rightarrow A_{1}, \gamma:A_{2}\times A_{2}\rightarrow A_{2}^{\ast}$ be the bilinear maps satisfying
	 	\begin{align*}
	 		\varphi(e_{1},e_{1}) &= \varphi(e_{2},e_{2}) = \gamma(e,e) = 0,\\
	 		\varphi(e_{1},e_{2}) &= -e^{\ast},
	 		\varphi(e_{2},e_{1}) = 2e^{\ast},\\
	 		\lambda(e,e) &= 2e_{2}.
	 	\end{align*}
	 	And let $v:A_{2}\times A_{1}\rightarrow A_{2}^{\ast}, v^{\prime}: A_{1}\times A_{2}\rightarrow A_{2}^{\ast}$ be the bilinear maps such that
	 	\begin{align*}
	 		v(e,e_{1}) = -4e^{\ast}, v^{\prime}(e_{1},e) = 2e^{\ast} \; \text{and} \; v(e,e_{2}) = v^{\prime}(e_{2},e) = 0.
	 	\end{align*}
	 	Then $(A_{2}\oplus A_{1}\oplus A_{2}^{\ast},\cdot,\hat{B}(\cdot,\cdot))$ is a 4-dimensional quadratic Novikov algebra whose basis is $\{e,e_{1},e_{2},e^{\ast}\}$.}
	 \end{ex}

Note that for a double extension $(A_{2}\oplus A_{1}\oplus A_{2}^{\ast}, B_{1}(\cdot,\cdot),\tau(\cdot,\cdot),\varphi,\mu,\mu^{\prime},v,v^{\prime},\lambda,\gamma)$, $A_{2}^{\ast}$ is an isotropic ideal of this double extension. Therefore, it is natural to consider whether any quadratic Novikov algebra having a nonzero  isotropic ideal  is isomorphic to some double extension.

\delete{Finally, we consider what kinds of quadratic Novikov algebras can be obtained from double extensions of quadratic Novikov algebras.}
	 \begin{lem}\label{ideal}
	 	Let $(A,\circ,B(\cdot,\cdot))$ be a quadratic Novikov algebra and $I$ be an ideal of $(A, \circ)$. Then  $(I + I^{\bot})/(I\cap I^{\bot})$ is a quadratic Novikov algebra with the bilinear form $\bar{B}(\cdot,\cdot)$ defined by $\bar{B}(\bar{x},\bar{y}) = B(x,y)$ for all $x,y \in I + I^{\bot}$.
	 \end{lem}
	 \begin{proof}
	 It is easy to see that $\bar{B}(\cdot,\cdot)$ is well-defined. Since $B(\cdot,\cdot)$ is symmetric and invariant, $\bar{B}(\cdot,\cdot)$ is also symmetric and invariant.
	 	
	 	Let $\bar{x}\in (I + I^{\bot})/(I\cap I^{\bot})$ such that $\bar{B}(\bar{x},\bar{y}) = 0$ for all $\bar{y} \in (I + I^{\bot})/(I\cap I^{\bot})$. Let $x = x_1 + x_2$ where $x_1\in I$ and $x_2 \in I^{\bot}$. Then we obtain $\bar{B}(\bar{x},\bar{a}) = B(x,a) = B(x_1 + x_2, a) =B(x_1,a) = 0$ for all $a\in I$. Hence $x_1 \in I^{\bot}$, that is $x_1\in I\cap I^{\bot}$. Thus $\bar{x_1} = \bar{0}$.
	 	Similarly, we also have $\bar{B}(\bar{x},\bar{b}) = B(x_2,b) = 0$ for all $b\in I^{\bot}$. Then $x_2\in I$, which means that $x_2\in I\cap I^{\bot}$ and $\bar{x_2} = \bar{0}$. Then $\bar{x} = \overline{x_1 + x_2} = \bar{0}$. Thus $\bar{B}(\cdot,\cdot)$ is nondegenerate. Therefore $(I + I^{\bot})/(I\cap I^{\bot})$ is a quadratic Novikov algebra.
	 \end{proof}
	 \begin{lem}\label{lem-1}
	 	Let $(A,\circ,B(\cdot,\cdot))$ be a quadratic Novikov algebra and $J$ be a nonzero isotropic ideal of $(A,\circ)$. Let $A = J^{\bot}\oplus V$ where $V$ is a subspace of $A$, and $S = J\oplus V$. Then $B(\cdot,\cdot)|_{S\times S}$ is nondegenerate and    $
	 		A = J\oplus S^{\bot}\oplus V.
	 	$
	 \end{lem}
	 \begin{proof}
	 	By  Lemma \ref{nonis}, $J^{\bot}$ is an ideal	of $(A, \circ)$. Since $J$ is an isotropic ideal of $(A,\circ)$, we have $J\subset J^{\bot}$ and $B(J,J)=0$.

Let $s=x+v \in S\cap S^{\bot}$ where $x\in J$ and $v\in V$. Then for all $t\in J$, we have $0=B(s,t)=B(x+v,t)=B(t,v)$. Therefore, $v\in J^{\bot}\cap V$. Since $J^{\bot}\cap V=0$, we get $v=0$. Then we obtain $B(x, S)=0$. Note that $B(x, J^{\bot})=0$. Therefore $B(x, S+J^{\bot})=0$. Hence we get $B(x, A)=0$. Since $B(\cdot,\cdot)$ is nondegenerate, we have $x= 0$, i.e. $s = 0$.  Therefore, $B(\cdot,\cdot)|_{S\times S}$ is nondegenerate.
Then by Lemma \ref{nonis}, we have $A = S\oplus S^{\bot} = J\oplus S^{\bot}\oplus V$. This completes the proof.
	 \end{proof}
\delete{
	 \begin{lem}
	 	Let $J$ be a maximal ideal of $A$ such that $A = J\oplus V$ where $V$ is a subspace of $A$. If $S = J^{\bot}\oplus V$, $B(\cdot,\cdot)$ restricted on $S$ is nondegenerate and
	 	\begin{align*}
	 		A = J^{\bot}\oplus S^{\bot}\oplus V.
	 	\end{align*}
	 \end{lem}
	 \begin{proof}
	 	Let $s=j+v \in S^{\bot} \cap S$ where $j\in J^{\bot}$ and $v\in V$. For $t = j^{\prime} + v^{\prime} \in S$, we have $B(j+v,j^{\prime} + v^{\prime}) = B(j,v^{\prime}) + B(v,j^{\prime}) + B(v,v^{\prime}) = 0$. We take $t = j^{\prime} \in J^{\bot}$, then $B(s,t) = B(v,j^{\prime}) = 0$. Since $j^{\prime}\in J^{\bot}$, we get $v\in J\cap V$, that is $v = 0$. Then we have $B(j,j^{\prime} + v^{\prime}) = 0$. And by $B(j,J) = 0$, we obtain that $B(j,j^{\prime} + v^{\prime} + J) = 0$, that is $B(j,A) = 0$. Since $B(\cdot,\cdot)$ is nondegenerate, $j = 0$ and $s = 0$. Thus $B(\cdot,\cdot)$ restricted on $S$ is nondegenerate. By Lemma \ref{nonis}, we have $A = S\oplus S^{\bot} = J^{\bot}\oplus S^{\bot}\oplus V$. This completes the proof.
	 \end{proof}}
\begin{rmk}
By Lemma \ref{lem-1}, we have $J\subset J^{\bot}$ and $S^{\bot}\subset J^{\bot}$. Therefore, $J\oplus S^{\bot}\subset J^{\bot}$. Since $A = J^{\bot}\oplus V$ and $A = J\oplus S^{\bot}\oplus V$, we have $J^{\bot} = J\oplus S^{\perp}$.
\end{rmk}
	 \begin{pro}\label{iso}
	 	Let $(A,\circ,B(\cdot,\cdot))$ be a quadratic Novikov algebra, $J$ be a nonzero isotropic ideal of $(A,\circ)$. Let $A = J^{\bot}\oplus V$ where $V$ is a subspace of $A$, and $S = J\oplus V$. Then $(W=J^{\bot}/J,\circ,\bar{B}(\cdot,\cdot))$ is a quadratic Novikov algebra with the bilinear form $\bar{B}(\cdot,\cdot)$ defined by $\bar{B}(\bar{x},\bar{y}) = B(x,y)$ for all $x,y \in J^{\bot}$ and there exists a quadratic Novikov algebra structure on $S^{\bot}$ which is isomorphic to $(W,\circ,\bar{B}(\cdot,\cdot))$.
	 \end{pro}
	 \begin{proof}
	 	Since $J^{\bot}/J = (J + J^{\bot})/(J \cap J^{\bot})$, we have $(J^{\bot}/J,\circ,\bar{B}(\cdot,\cdot))$ is a quadratic Novikov algebra by Lemma \ref{ideal}.
	 	
	 	Let $x$, $y\in S^{\bot}$. Since $S^{\bot}\subset J^{\bot}$, we assume that $x\circ y = \varphi(x,y) + x\diamond y$, where $\varphi(x, y)\in J$ and $x\diamond y\in S^{\perp}$. Since $J$ is an ideal of $J^{\bot}$, it is easy to check that $(S^{\perp}, \diamond)$ is a Novikov subalgebra. Since $B(\cdot,\cdot)\mid_{S\times S}$ is nondegenerate, we have that $\hat{B}(\cdot,\cdot)=B(\cdot,\cdot)\mid_{S^{\bot}\times S^{\bot}}$ is nondegenerate by Lemma  \ref{nonis}. Moreover, since $B(\cdot,\cdot)$ is invariant and symmetric and $B(S^{\bot}, J)=0$, then $\hat{B}(\cdot,\cdot)$ is also invariant and symmetric on $(S^{\bot}, \diamond)$. Hence $(S^{\bot},\diamond,\hat{B}(\cdot,\cdot))$ is a quadratic Novikov algebra.
	 	
	 	Let $\theta: S^{\bot}\rightarrow W$ be a linear map defined by $\theta(x) = \bar{x}$ for all $x\in S^{\perp}$. If $\bar{x} = 0$, we have $x\in J$. Since $x\in S^{\perp}$, we get $x = 0$. Thus $\theta$ is injective. Note that $\text{dim} W= \text{dim} S^{\perp}$. Therefore, $\theta$ is a bijection. Moreover, since
	 	\begin{align*}
	 		\theta(x\diamond y) = \overline{x\diamond y} = \overline{\varphi(x, y) + x\diamond y} = \overline{x\circ y} = \bar{x}\circ \bar{y} = \theta(x)\circ \theta(y),\;\; x, y\in S^{\bot},
	 	\end{align*}
	 	$\theta$ is an isomorphism of Novikov algebras. It is easy to see that $\bar{B}(\theta(x),\theta(y))=\hat{B}(x,y)$ for all $x$, $y\in S^{\bot}$. Therefore, $\theta$ is an isomorphism of quadratic Novikov algebras.		
	 \end{proof}

	 \begin{thm}\label{thm-4}
	 	Let $(A,\circ,B(\cdot,\cdot))$ be a quadratic Novikov algebra and $J$ be a nonzero isotropic ideal of $(A,\circ)$. Let $A = J^{\bot}\oplus V$ where $V$ is a subspace of $A$,  $S = J\oplus V$ and $(W=J^{\bot}/J,\circ,\bar{B}(\cdot,\cdot))$ be the quadratic Novikov algebra given in Proposition \ref{iso}. Then the quadratic Novikov algebra $(A = V\oplus S^{\bot}\oplus J, \circ,B(\cdot,\cdot))$ is isomorphic to a double extension of $(W=J^{\bot}/J,\circ, \bar{B}(\cdot,\cdot))$ by a Novikov algebra $(V, \bullet)$ with a symmetric bilinear form $B(\cdot,\cdot)|_{V\times V}$.
	 \end{thm}
	 \begin{proof}
Since $A=J^{\bot}\oplus V$, for all $u$, $v\in V$, we can set $u\circ v=\alpha(u, v)+u\bullet v$, where $\alpha(u, v)\in J^{\bot}$ and $u\bullet v\in V$. It is direct to check that $(V, \bullet)$ is a Novikov algebra. Let $L_{V,\star}=L_\bullet +R_\bullet$. Set $\tilde{B}(\cdot,\cdot)=B(\cdot,\cdot)|_{V\times V}$ and $\hat{B}(\cdot,\cdot)=B(\cdot,\cdot)|_{S^{\bot}\times S^{\bot}}$.

Let $\Phi: J \rightarrow V^{\ast}$ be the linear map defined as $\Phi(x)(u) = B(x,u)$ for all $x\in J$ and $u\in V$. Since $A=J^{\bot}\oplus V$ and $B(\cdot,\cdot)$ is nondegenerate, we get that $\varphi$ is a bijection.

Let $a\in A$, $x\in J^{\bot}$ and $y\in J$. Since $B(a,x\circ y) = -B(y,x\star a) \in B(J, J^{\bot}) = 0$ for all $a\in A$, we get $J^{\bot}\circ J = 0$, by the nondegenerate property of $B(\cdot,\cdot)$. Similarly, $J\circ J^{\bot}=0$. Therefore, we obtain that $(V\circ V)\circ J=(V \bullet V)\circ J$ and $J \circ (V\circ V)=J \circ(V\bullet V)$. Since $J$ is an ideal of $(A, \circ)$, $(J, L_\circ, R_\circ)$ is also a representation of $(V, \bullet)$. Let $u$, $v\in V$, $x\in J$. We obtain that
	 	\begin{eqnarray*}
	 		\Phi(L_\circ (u)x)(v)&=&\Phi(u\circ x)(v) = B(u\circ x,v) = -B(x,u\circ v+v\circ u) \\
&=&-B(x,\alpha(u, v)+\alpha(v, u)+L_{V,\star}(v)w) \\
&=&-B(x, L_{V,\star}(v)w)= -\Phi(x)(L_{V,\star}(v)w) = (L^{\ast}_{V,\star}(v)\Phi(x))w,\\
	 		\Phi(R_\circ(u)x)(v)&=& \Phi(x\circ u)(v) = B(x\circ u,v) = B(v\circ u,x) = B(\alpha(v, u)+v\bullet u, x)=B(v\bullet u, x)\\
 &=&B(R_{\bullet}(u)v,x) = \Phi(x)(R_{\bullet}(v)u) = (-R^{\ast}_{\bullet}(v)\Phi(x))u.
	 	\end{eqnarray*}
 Therefore $\Phi$ is an isomorphism of representations of $(V, \bullet)$ from $(J, L_\circ, R_\circ)$ to $(V^\ast, L^{\ast}_{V,\star}, -R^{\ast}_{\bullet})$.

 Let $a = u+ y_1 + x_1$ and $b = v+y_2 + x_2\in A=V \oplus S^{\bot}\oplus J$, where $u$, $v\in V$, $x_1$, $x_2\in J$ and $y_1$, $y_2\in S^{\bot}$. Note that $J$ is an ideal of $(A,\circ)$, and $J^{\bot}\circ J=J\circ J^{\bot}=0$.
	 	In addition, since $S=J\oplus V$ and $J^{\bot}=J\oplus S^{\bot}$, we have $J\circ J=J\circ S^{\bot}= S^{\bot}\circ J=0$ and $S^{\bot} \circ S^{\bot}\subset J^{\bot}\circ J^{\bot}\subset J\oplus S^{\bot}$. Moreover, $V\circ S^{\bot}\subset V\circ J^{\bot}\subset S^{\bot}\oplus J$ and $S^{\bot} \circ V\subset J^{\bot}\circ V\subset S^{\bot}\oplus J$. Therefore, we can assume that
	 	\begin{align*}
	 		a\circ b &= (u + y_1 + x_1)\circ(v + y_2 + x_2 ) \\&= u\circ v + u\circ y_2+ u\circ x_2 + y_1\circ v + y_1\circ y_2 + x_1\circ v \\
&= u\bullet v + \lambda(u,v) + \gamma(u,v) + \mu(u)(y_2)+ h(u,y_2) + L_\circ(u) x_2 + y_1\diamond y_2 + \varphi(y_1,y_2) \\
&\quad+ \mu^{\prime}(v)(y_1)+ h^{\prime}(y_1,v) + R_\circ (v)x_{1},
	 	\end{align*}
	 	where $\lambda:V\times V \rightarrow S^{\bot}$, $\gamma:V\times V \rightarrow J$, $h: V\times S^{\bot} \rightarrow J$, $h^{\prime}: S^{\bot}\times V \rightarrow J$ and $\varphi: S^{\bot}\times S^{\bot}\rightarrow J$ are bilinear maps, and $\mu$, $\mu^{\prime}: V\rightarrow \text{End}_{\bf k}(S^{\bot})$ are linear maps. Moreover, we obtain
\begin{eqnarray*}
B(a, b)&=&B(u + y_1 + x_1,v + y_2 + x_2)\\
&=&B(u, v)+B(u,x_2)+B(y_1,y_2)+B(x_1, v)\\
&=&\tilde{B}(u,v)+\hat{B}(y_1,y_2)+\Phi(x_2)u+\Phi(x_1)v.
\end{eqnarray*}

Let $\sigma: A=V \oplus S^{\bot}\oplus J\rightarrow V\oplus W\oplus V^{\ast}$ be the linear map defined by $\sigma(u + y + x) = u + \theta(y) + \Phi(x)$, where $x\in J$, $y\in S^{\perp}$, $u\in V$ and $\theta: S^{\bot} \rightarrow W$ is the linear map defined in the proof of Proposition \ref{iso}. Since $\theta$ and $u$ is bijective, $\tau$ is a linear isomorphism. Then it is easy to see that $(V\oplus W\oplus V^\ast, \bar{B}(\cdot,\cdot), \tilde{B}(\cdot,\cdot), \Phi\circ \varphi\circ (\theta^{-1}\times \theta^{-1}), \mu_1, \mu_1^{\prime}, \Phi\circ h\circ (\id \times \theta^{-1}), \Phi\circ h^{\prime}\circ (\theta^{-1}\times \id), \theta\circ \lambda, \Phi\circ \gamma)$ is a double extension of
of $(W=J^{\bot}/J,\circ, \bar{B}(\cdot,\cdot))$ by the Novikov algebra $(V, \bullet)$ with a symmetric bilinear form $\tilde{B}(\cdot,\cdot)$, where $\mu_1(u)=\theta\circ \mu(u)\circ\theta^{-1}$ and $\mu_1^{\prime}(u)=\theta\circ \mu^{\prime}(u)\circ\theta^{-1}$ for all $u\in V$. Moreover, one can check that $\sigma$ is an isomorphism of quadratic Novikov algebras from $(A, \circ, B(\cdot,\cdot))$ to $(V\oplus W\oplus V^\ast, \bar{B}(\cdot,\cdot), \tilde{B}(\cdot,\cdot), \Phi\circ \varphi\circ (\theta^{-1}\times \theta^{-1}), \mu_1, \mu_1^{\prime}, \Phi\circ h\circ (\id \times \theta^{-1}), \Phi\circ h^{\prime}\circ (\theta^{-1}\times \id), \theta\circ \lambda, \Phi\circ \gamma)$.  This completes the proof.
	 \end{proof}

\begin{rmk}
We consider whether all nontrivial quadratic Novikov algebras of dimension $3$ are obtained from some double extensions of a $1$-dimensional trivial quadratic Novikov algebra by a $1$-dimensional Novikov algebra with a symmetric bilinear form. By Theorem \ref{3-dim}, we only need to consider the three cases as follows.
\begin{enumerate}
\item For Type $(A7) (l=-2)$, $J=\mathbb{C}e_1$ is an isotropic ideal. Then $J^\bot=\mathbb{C} e_1\oplus \mathbb{C} e_2$, $V=\mathbb{C} e_3$, $S=\mathbb{C}e_1 \oplus \mathbb{C} e_3$ and $S^\bot=\mathbb{C} e_2$;
\item For Type $(C5) (l=-2)$, $J=\mathbb{C}e_1$ is an isotropic ideal. Then
$J^\bot=\mathbb{C} e_1\oplus \mathbb{C} e_2$, $V=\mathbb{C} e_3$, $S=\mathbb{C}e_1 \oplus \mathbb{C} e_3$ and $S^\bot=\mathbb{C} e_2$;
\item For Type $(D6) (l=-\frac{1}{2})$, $J=\mathbb{C}e_2$ is an isotropic ideal. Then
$J^\bot=\mathbb{C} e_1\oplus \mathbb{C} e_2$, $V=\mathbb{C} e_3$, $S=\mathbb{C}e_2 \oplus \mathbb{C} e_3$ and $S^\bot=\mathbb{C} e_1$.
\end{enumerate}
Therefore by Theorem \ref{thm-4}, all nontrivial quadratic Novikov algebras of dimension $3$ are obtained from some double extensions of a $1$-dimensional trivial quadratic Novikov algebra by a $1$-dimensional Novikov algebra with a symmetric bilinear form.
\end{rmk}
	
	\delete{
	\begin{ex}
		Let $(A = V\oplus S^{\bot}\oplus J = {\bf k}v + {\bf m}s + {\bf n}j,\circ,B(\cdot,\cdot))$ be the 3-dimensional quadratic Novikov algebra of case (A7)($ l = -2 $) in [Table 2].
		Then $A$ is isomorphic to the double extension $V\oplus W\oplus V^{\ast} = {\bf k}v\oplus {\bf m}w \oplus {\bf n}v^{\ast}$ of $W$ by $V$, where $w\circ w = 0,\; v\circ v = \lambda(v,v) = w,\; \mu(v) = \mu^{\prime}(v) = 0,\; v^{\prime}(w,v) = v^{\ast},\; v(v,w) = -2v^{\ast},\; v^{\ast}(v) = B(v,j)$.
	\end{ex}}
	
\noindent {\bf Acknowledgments.} This research is supported by
NSFC (No. 12171129) and  the Zhejiang
Provincial Natural Science Foundation of China (No. Z25A010006).

\smallskip

\noindent
{\bf Declaration of interests. } The authors have no conflicts of interest to disclose.

\smallskip

\noindent
{\bf Data availability. } No new data were created or analyzed in this study.

\UseRawInputEncoding

\end{document}